\documentclass[a4paper,11pt,twoside]{article}
\usepackage{geometry}
\usepackage{hyperref}
\usepackage{graphicx}
\usepackage{subfigure}
\usepackage{float}
\usepackage{epstopdf}
\usepackage{color}
\usepackage{amsmath,amsthm}
\usepackage{amssymb}
\usepackage{amsfonts}
\usepackage[noend]{algpseudocode}
\usepackage{algorithmicx,algorithm}
\usepackage{cite}
\usepackage{tikz}
\usepackage{dsfont}

\newtheorem{thm}{Theorem}[section]
\newtheorem{cor}{Corollary}[section]
\newtheorem{lem}{Lemma}[section]

\newtheorem{exa}{Example}[section]
\newtheorem{defi}{Definition}[section]
\theoremstyle{remark}
\newtheorem{rem}{Remark}[section]
\numberwithin{equation}{section}

\newcommand{\bu}{{\bf u}}
\newcommand{\bv}{{\bf v}}
\newcommand{\diag}{{\rm diag}}
\newcommand{\rank}{{\rm rank}}
\newcommand{\range}{{\rm range}}
\newcommand{\R}{{\mathbb R}}
\newcommand{\Rmn}{{\mathbb R}^{m\times n}}

\newcommand{\BE}{\begin{equation}}
\newcommand{\EE}{\end{equation}}
\begin{document}
%---------------------------------------------
\title{Single-pass randomized QLP decomposition for low-rank approximation}

\author {Huan Ren\thanks{School of Mathematical Sciences, Xiamen University, Xiamen 360015, People's Republic of China (jxrh1994@163.com)}
\and Zheng-Jian Bai\thanks{School of Mathematical Sciences and Fujian Provincial Key Laboratory on Mathematical Modeling \& High Performance Scientific Computing,  Xiamen University, Xiamen 361005, People's Republic of China (zjbai@xmu.edu.cn). The research of this author was partially supported by the National Natural Science Foundation of China (No. 11671337) and the Fundamental Research Funds for the Central Universities (No. 20720180008).} }
\maketitle
%---------------------------------------------
\begin{abstract}
\noindent%\noindent: ÁÙÊ±È¡ÏûÊ×ÐÐËõœø
The QLP decomposition is one of the effective algorithms to approximate singular value decomposition (SVD) in numerical linear algebra. In this paper, we propose some single-pass randomized QLP decomposition algorithms for computing the low-rank matrix approximation. Compared with the deterministic QLP decomposition, the complexity of the proposed algorithms does not increase significantly and the system matrix needs to be accessed only once. Therefore, our algorithms are very suitable for a large matrix stored outside of memory or generated by stream data. In the error analysis, we give the bounds of matrix approximation error and singular value approximation error. Numerical experiments also reported to verify our results.

\end{abstract}
%---------------------------------------------
\section{Introduction}
Let $A\in\mathbb{R}^{m\times n}$ be a data matrix. The  low-rank approximation of $A$ is to compute two low-rank matrices $E\in\mathbb{R}^{m\times k}$ and $F\in\mathbb{R}^{k\times n}$ such that
\begin{equation}\label{equ:low-rank}
A\approx EF,
\end{equation}
where $\rank(E)=\rank(F)=k$, and $k\ll\min\{m,n\}$, the rank $k$ is given to us in advance.

In the era of big data, the data we deal with is often extremely large. In other words, the scale of the data matrix is very large. In this case, the low-rank approximation in the form of (\ref{equ:low-rank}) can greatly reduce the storage of the data matrix $A$ (i.e., we only need to store $E$ and $F$ instead of $A$). Low-rank approximation is one of the essential tools in scientific computing, including principal component analysis \cite{Jolliffe1986PCA,Rokhlin2009PCA,Halko2011PCA,Feng2018PCA}, data analysis \cite{Mahoney2011data,Drineas2016RNLA}, and fast approximate algorithms for PDEs \cite{Martinsson2011HSS,Ghysels2016HSS,Xia2010Cholesky}.

For a general matrix $A\in\mathbb{R}^{m\times n}$, we usually consider an rank-$k$  approximate singular value decomposition (SVD), i.e.,
\begin{equation}
A\approx U_{k}\Sigma_{k}V_{k}^{T},
\end{equation}
where $U_{k}=[\bu_{1},\bu_{2},\ldots,\bu_{k}]$, $\Sigma_{k}=\diag(\sigma_{1}(A),\sigma_{2}(A),\ldots,\sigma_{k}(A))$, and $V_{k}=[\bv_{1},\bv_{2},\ldots,\bv_{k}]$ with $\sigma_{1}(A)\geq\sigma_{2}(A)\geq\cdots\geq\sigma_{k}(A)\ge 0$, $\bu_{j}$ and $\bv_{j}$ are the left and right singular vectors corresponding to $\sigma_{j}(A)$, respectively. Such low-rank approximation is optimal as stated as follows:

\begin{thm}\cite{Goulb2013MC,Eckart1936low-rank}\label{thm:Eckart}
Given $A\in\mathbb{R}^{m\times n}$, for any integer $k$ with $1\leq k<r=\rank(A)$, define
$$A_{k}=\sum_{i=1}^{k}\sigma_{i}\bu_{i}\bv_{i}^{T}.$$
Then $$\|A-A_{k}\|_{2}=\min_{B\in\mathbb{R}^{m\times n} \atop \rank(B)\leq k}\|A-B\|_{2}=\sigma_{k+1},$$
and $$\|A-A_{k}\|_{F}=\min_{B\in\mathbb{R}^{m\times n} \atop \rank(B)\leq k}\|A-B\|_{F}=\sqrt{\sigma_{k+1}^{2}+\cdots+\sigma_{r}^{2}}.$$
\end{thm}

Theorem \ref{thm:Eckart} shows that the rank-$k$ truncated SVD provides the smallest error for the rank-$k$ approximation of $A$. Therefore, the truncated SVD is the best low-rank approximation with a given fixed rank. However, the computation of a SVD of a large matrix $A$ is very costly. Therefore, we wish to find an algorithm for computing a low-rank approximation to a large matrix. As expected, we hope the proposed algorithm is close to the quality that the SVD provides but needs  much lower cost.

The QLP decomposition was proposed by Stewart in 1999 \cite{Stewart1999QLP}, which can be regarded as an economical method for computing an approximate SVD. In fact, the QLP decomposition is equivalent to two consecutive QR decomposition with column pivoting (QRCP). Specifically, the QRCP is performed  on the data matrix $A$ in the sense that
\BE\label{qrcp}
AP_{0}=Q_{0}R_{0}\quad \mbox{and}\quad R_{0}^{T}P_{1}=Q_{1}L^{T},
\EE
where $P_{0}\in\R^{n\times n}$ and $P_{1}\in\R^{m\times m}$ two permutation matrices, $Q_{0}\in\R^{m\times m}$ and $Q_{1}\in\R^{n\times n}$ are two orthogonal matrices and $L\in\mathbb{R}^{m\times n}$ is a lower triangular matrix. The diagonal elements of $R_{0}$ are called the {\em $R$-values} and the diagonal elements of $L$ are called the {\em $L$-values}. Define $Q=Q_{0}P_{1}$, $P=P_{0}Q_{1}$, then $$A=Q_{0}P_{1}LQ_{1}^{T}P_{0}^{T}:=QLP^{T}.$$
Huckbay and Chan \cite{Huckaby2003QLP} showed that the $L$-values approximate the singular values of the original matrix $A$ with considerable fidelity. The truncated QLP decomposition of $A$ can be expressed as follows:
\begin{equation}\label{equ:TQLP}
A\approx Q_{k}L_{k}P_{k}^{T}:=A_{k},
\end{equation}
where both $Q_{k}\in\R^{m\times k}$ and $P_{k}\in\R^{n\times k}$ have orthonormal column vectors and $L_{k}\in\R^{k\times k}$ is lower triangular. The truncated QLP decomposition (\ref{equ:TQLP}) can also be regarded as a low-rank approximation of $A$. It is natural to expect the truncated QLP decomposition performs as the truncated SVD.

In recent years, randomized algorithms for low-rank approximation have attracted considerable attention \cite{Kaloorazi2018subspace-orbit,Shabat2018RLU,Guming2015Subspaceiter,Halko2011SIAMreview}. Compared with deterministic algorithms, randomized algorithms for low-rank approximation have the advantages of low complexity, fast running speed and easy implementation. However, these randomized algorithms need to access the original matrix $A$ at least twice, which is expensive for the large matrix stored outside of core memory or generated by stream data.

As we know, the cost of data communication is often much higher than the algorithm itself. In order to reduce the cost of data communication, some single-pass algorithms have been proposed \cite{Halko2011SIAMreview,Li2020single-pass,Bjarkason2019pass-efficient,Tropp2019streaming,Tropp2017sketching,Woodruff2014sketching}. In this paper, based on the idea of single-pass, we extend the work of Wu and Xiang \cite{Xiang2020RQLP} to the  single-pass randomized QLP decomposition for computing low-rank approximation, where  two randomized algorithms are provided. We also give the  bounds of matrix approximation error and singular value approximation error for the proposed randomized algorithms, which hold with high probability.

The rest of this paper is organized as follows. In Section 2 we give some preliminary  results related to subgaussian random matrices and some basic QLP decomposition algorithms. In Section 3 we propose two single-pass randomized QLP decomposition algorithms and present the corresponding  complexity analysis. In Section 4 we give the error analysis of the proposed randomized algorithms, including  matrix approximation error and singular value approximation error. Finally, some numerical examples and concluding remarks  are given  in Section 5 and Section 6, respectively.
%--------------------------------------------------------------------
\section{Preliminaries}
In this section, we review some preliminary lemmas on  subgaussian random matrices and some basic QLP decomposition algorithms.

In this paper, we use the following notations.  Let $\Rmn$ be the set of all $m\times n$ real matrices.
For any  $A\in \mathbb{R}^{m\times n}$, let $\sigma_{1}(A)\geq\sigma_{2}(A)\geq\cdots\geq\sigma_{q}(A)\ge 0$ denote the singular values of $A$, where $q=\min\{m,n\}$. Denote by $\|\cdot\|_{2}$, $\|\cdot\|_{F}$ the matrix $2$-norm and the matrix Frobenius norm, respectively. Let $A^T$ and  $A^{\dag}$ denote the transpose and the Moore-Penrose inverse of a matrix $A$ , respectively. In addition, $\mathbb{E}(\cdot)$ denotes expectation of random variable and $\mathbb{P}(\cdot)$ denotes probability of random event.
\subsection{Subgaussian matrix}
In this subsection, we recall some preliminary lemmas on subgaussian random matrices.
\begin{defi}
A random variable $X$ is called subgaussian if there exist constants $\beta,\kappa>0$ such that $$\mathbb{P}(|X|\geq t)\leq\beta e^{-\kappa t^{2}}~~~~\textrm{for all}~ t>0.$$
\end{defi}

\begin{defi}\cite{Litvak2005smallest}\label{def:general-gaussian}
Assume $\mu\geq1$ and $a_{1},a_{2}>0$. Let $\mathcal{A}(m,n,\mu,a_{1},a_{2})$ be the set of all $m\times n$ $(m>n)$ random matrices $A=(\xi_{ij})$ whose entries are centered independent identical distribution (i.i.d.) real-valued random variables satisfying the following conditions:
\begin{enumerate}
\item[1).] Moments: $\mathbb{E}(|\xi_{ij}|^{3})\leq\mu^{3}$;

\item[2).]  Norm: $\mathbb{P}(\|A\|_{2}>a_{1}\sqrt{m})\leq e^{-a_{2}m}$;

\item[3).] Variance: $\mathbb{E}(\xi_{ij}^{2})\geq1$.
\end{enumerate}
\end{defi}
\begin{rem}\label{rem:parameter0}
It is easy to find that subgaussian matrices and Gaussian matrices  are random matrices defined by Definition \ref{def:general-gaussian}. Specifically, if $A$ is subgaussian, then $A\in\mathcal{A}(m,n,\mu,a_{1},a_{2})$ and $a_{1}=6\mu\sqrt{a_{2}+4}$; if $A$ is a standard Gaussian random matrix, then $A\in\mathcal{A}(m,n,\mu,a_{1},a_{2})$ and $\mu=\left(\frac{4}{\sqrt{2\pi}}\right)^{\frac{1}{3}}$.
\end{rem}

The following lemma provides a lower bound of the smallest singular value of a randomized matrix, which holds with high probability.

\begin{lem}\cite{Litvak2005smallest}\label{lem:smallest-bound}
Let $\mu\geq1$, $a_{1},a_{2}>0$. Let $A\in\mathcal{A}(m,n,\mu,a_{1},a_{2})$ with $m=(1+\delta)n>(1+\frac{1}{\ln n})n$, where $\delta>0$. Then, there exist two positive constants $c_{1},c_{2}$ such that
\begin{equation}
\mathbb{P}(\sigma_{n}(A)\leq c_{1}\sqrt{m})\leq e^{-c_{2}m}.
\end{equation}
\end{lem}

\begin{rem}\label{rem:parameter}
From \cite{Litvak2005smallest}, the exact value of $c_{1},c_{2}$ in Lemma \ref{lem:smallest-bound} are $$c_{1}=\frac{b}{e^{2}c_{3}}\left(\frac{b}{3e^{2}c_{3}a_{1}}\right)^{\frac{1}{\delta}},~ c_{2}=\min\left\{1,\frac{c^{''}}{2\mu^{6}},a_{2}\right\}-\frac{\ln3}{m},$$
where $c_{3}=4\sqrt{\frac{2}{\pi}}\left(\frac{2\mu^{9}}{a_{1}^{3}}+\sqrt{\pi}\right)$, $b=\min\left(\frac{1}{4},\frac{c^{'}}{5a_{1}\mu^{3}}\right)$, $c^{'}=\left(\frac{27}{2^{13}}\right)^{\frac{1}{2}}$ and $c^{''}=\frac{27}{2^{11}}$.
\end{rem}

\subsection{Randomized QLP decomposition}
In this subsection, we recall the  QLP decomposition and the randomized QLP decomposition. We first recall  the  QLP decomposition \cite{Stewart1999QLP}.
\begin{algorithm}\label{algor:QLP}
\caption{The QLP decomposition\cite{Stewart1999QLP}} % Ëã·šµÄÃû×Ö
\hspace*{0.02in} {\bf Input:} %Ëã·šµÄÊäÈë£¬ \hspace*{0.02in}ÓÃÀŽ¿ØÖÆÎ»ÖÃ£¬Í¬Ê±ÀûÓÃ \\ œøÐÐ»»ÐÐ
$A\in\mathbb{R}^{m\times n}$.\\
\hspace*{0.02in} {\bf Output:} %Ëã·šµÄœá¹ûÊä³ö
Matrices $Q,L,P$ such that $A= QLP^{T}$, where $Q$ is a column orthogonal matrix, \\$P$ is a orthogonal matrix and $L$ is a lower triangular matrix.

\hspace*{0.02in} {\bf function} $[Q,L,P]=$QLP$(A)$
\begin{algorithmic}[1]
\State $[Q_{0},R_{0},P_{0}]=\textrm{QRCP}(A)$.
\State $[Q_{1},L^{T},P_{1}]=\textrm{QRCP}(R_{0}^{T})$.
\State $Q=Q_{0}P_{1}$, $P=P_{0}Q_{1}$.
%\For{condition} % For ÓïŸä£¬ÐèÒªºÍEndFor¶ÔÓŠ
%¡¡¡¡\State ...
%¡¡¡¡\If{condition} % If ÓïŸä£¬ÐèÒªºÍEndIf¶ÔÓŠ
%¡¡¡¡¡¡¡¡\State ...
%¡¡¡¡\Else
%¡¡¡¡¡¡¡¡\State ...
%¡¡¡¡\EndIf
%\EndFor
%\While{condition} % WhileÓïŸä£¬ÐèÒªºÍEndWhile¶ÔÓŠ
%¡¡¡¡\State ...
%\EndWhile
%\State \Return result
\end{algorithmic}
\end{algorithm}

The {\tt MATLAB} pseudo-code of rank-$k$ randomized  QLP decomposition algorithm is described as in Algorithm 2 \cite{Xiang2020RQLP}.
\begin{algorithm}\label{algor:RQLP}
\caption{Randomized QLP decomposition(RQLP)\cite{Xiang2020RQLP}} % Ëã·šµÄÃû×Ö
\hspace*{0.02in} {\bf Input:} %Ëã·šµÄÊäÈë£¬ \hspace*{0.02in}ÓÃÀŽ¿ØÖÆÎ»ÖÃ£¬Í¬Ê±ÀûÓÃ \\ œøÐÐ»»ÐÐ
$A\in\mathbb{R}^{m\times n}$, target rank:  $k\geq2$, oversampling parameter: $p\geq2$, and number of columns sampled: $l=k+p$.\\
\hspace*{0.02in} {\bf Output:} %Ëã·šµÄœá¹ûÊä³ö
Matrices $Q,L,P$ such that $A\approx QLP^{T}$, where $Q$ is a column orthogonal matrix, \\$P$ is a orthogonal matrix and $L$ is a lower triangular matrix.
\begin{algorithmic}[1]
\State $\Omega={\tt randn}(n,l)$. % \State ºóÐŽÒ»°ãÓïŸä
\State $Y=A\Omega$.
\State $[V,R]={\tt qr}(Y,0)$.
\State $B=V^{T}A$.
\State $[\widehat{Q},L,P]={\rm QLP}(B)$.
\State $Q=V\widehat{Q}$.
%\For{condition} % For ÓïŸä£¬ÐèÒªºÍEndFor¶ÔÓŠ
%¡¡¡¡\State ...
%¡¡¡¡\If{condition} % If ÓïŸä£¬ÐèÒªºÍEndIf¶ÔÓŠ
%¡¡¡¡¡¡¡¡\State ...
%¡¡¡¡\Else
%¡¡¡¡¡¡¡¡\State ...
%¡¡¡¡\EndIf
%\EndFor
%\While{condition} % WhileÓïŸä£¬ÐèÒªºÍEndWhile¶ÔÓŠ
%¡¡¡¡\State ...
%\EndWhile
%\State \Return result
\end{algorithmic}
\end{algorithm}

%---------------------------------------------
\section{Single-pass randomized QLP decomposition}
In this section, we present two  single-pass randomized QLP decomposition algorithms.
\subsection{Regular single-pass randomized QLP decomposition}\label{subsect:SPRQLP}
In this subsection, we give a regular single-pass randomized QLP decomposition algorithm for computing the low-rank approximation to a data matrix $A\in\Rmn$. To calculate the low-rank approximation of $A$, we first construct a low-rank matrix $V$ with orthonormal columns such that $A\approx VV^{T}A$ and  $\range(V)\approx\range(A)$. We observe that, in Algorithm 2, $\range(V)\approx\range(A)$ and $VV^{T}$ is an approximate orthogonal projector on $\range(A)$. Thus
\begin{equation}\label{equ:orh-project}
A\approx VV^{T}A.
\end{equation}

A single-pass randomized algorithm should realize that each entry of the input matrix can only be accessed once. To do so, we wish replace the matrix $B=V^{T}A$ in Step 4 of Algorithm 2 by another expression without $A$.
We note that, for the matrix $V$ generated by Algorithm 2,  we have $A\approx VV^{T}A$.  Then, for $B=V^{T}A$, we have $A\approx VB$. Premultiplying $\Omega_{2}$  on both sides of $A\approx VB$ we get
$$Y_{2}=\Omega_{2}A\approx\Omega_{2}VB.$$
Therefore, the matrix $B$ can be approximately expressed as $$B\approx(\Omega_{2}V)^{\dag}Y_{2}.$$

Next, we give the single-pass randomized QLP decomposition algorithm, which is stated in Algorithm 3.

\begin{algorithm}\label{algor:SPRQLP}
\caption{Single-pass randomized QLP decomposition (SPRQLP)} % Ëã·šµÄÃû×Ö
\hspace*{0.02in} {\bf Input:} %Ëã·šµÄÊäÈë£¬ \hspace*{0.02in}ÓÃÀŽ¿ØÖÆÎ»ÖÃ£¬Í¬Ê±ÀûÓÃ \\ œøÐÐ»»ÐÐ
$A\in\mathbb{R}^{m\times n}$, target rank:  $k\geq2$, oversampling parameter: $p\geq2$, number of columns sampled: $l_{1}=k+p$, and number of rows sampled: $l_{2}\geq l_{1}$.\\
\hspace*{0.02in} {\bf Output:} %Ëã·šµÄœá¹ûÊä³ö
Matrices $Q,L,P$ such that $A\approx QLP^{T}$, where $Q$ is a column orthogonal matrix, \\$P$ is a orthogonal matrix and $L$ is a lower triangular matrix.
\begin{algorithmic}[1]
\State $\Omega_{1}=\texttt{randn}(n,l_{1})$, $\Omega_{2}=\texttt{randn}(l_{2},m)$. % \State ºóÐŽÒ»°ãÓïŸä
\State $Y_{1}=A\Omega_{1}$, $Y_{2}=\Omega_{2}A$.
\State $[V,R]=\texttt{qr}(Y_{1},0)$.
\State $B=(\Omega_{2}V)^{\dag}Y_{2}$.
\State $[\widehat{Q},L,P]=\textrm{QLP}(B)$.
\State $Q=V\widehat{Q}$.
\end{algorithmic}
\end{algorithm}

%\subsubsection{Computational complexity analysis}
On  the complexity of Algorithm 3, we have the following remarks.

\begin{itemize}
\item{Step 1:} Generating random matrices $\Omega_{1}, \Omega_{2}$ takes $\mathcal{O}(nl_{1}+ml_{2})$ operations;
\item{Step 2:} Computing $Y_{1}=A\Omega_{1}$ and $Y_{2}=\Omega_{2}A$ takes $\mathcal{O}(mn(l_{1}+l_{2}))$ operations;
\item{Step 3:} Computing unpivoted QR decomposition of $Y_{1}$ of size $m\times l_{1}$, takes $\mathcal{O}(ml_{1}^{2})$ operations;
\item{Step 4:} Computing $\Omega_{2}V$ takes $\mathcal{O}(ml_{1}l_{2})$ operations, computing Moore-Penrose inverse of $\Omega_{2}V$ takes $\mathcal{O}(l_{1}^{2}l_{2}+l_{1}^{3}+l_{1}^{2}l_{2})=\mathcal{O}(l_{1}^{2}l_{2}+l_{1}^{3})$ operations, and multiplying it by $Y_{2}$ takes $\mathcal{O}(nl_{1}l_{2})$ operations;
\item{Step 5:} Computing column pivoting QR decomposition of $B$ of size $l_{1}\times n$, takes $\mathcal{O}(nl_{1}^{2})$ operations;
\item{Step 6:} Computing column pivoting QR decomposition of $R_{0}^{T}$ of size $n\times l_{1}$, takes $\mathcal{O}(nl_{1}^{2})$ operations;
\item{Step 7:} Computing $VQ_{0}P_{1}$ takes $\mathcal{O}(ml_{1}^{2})$ operations, computing $P_{0}Q_{1}$ takes $\mathcal{O}(n^{3})$ operations.
\end{itemize}

The total complexity of Algorithm 3 is $$\mathcal{C}_{SPRQLP}=\mathcal{O}\left(nl_{1}+ml_{2}+mn(l_{1}+l_{2})+(m+n)l_{1}l_{2}+(m+n+l_{2})l_{1}^{2}+l_{1}^{3}+n^{3}\right).$$
It is easy to see that the total complexity of Algorithm 2 is
$$\mathcal{C}_{RQLP}=\mathcal{O}\left(nl+mnl+(m+n)l^{2}+l^{3}+n^{3}\right).$$

We note that $l=l_{1}$ and $k\leq l_{1}\leq l_{2}\leq\min\{m,n\}$. Hence, Algorithm 2 has  slightly lower complexity than Algorithm 3. As we know, the cost of data communication is even more expensive than the algorithm itself. In particular, when the data is stored outside the core memory and the data matrix is very large, the cost of data communication is much larger than the algorithm itself. We observe that Algorithm 2 needs to access the data matrix $A$ twice. Thus, the total computation cost of Algorithm 2 may be much more expensive than Algorithm 3.

\subsection{Subspace-orbit single-pass randomized QLP decomposition}
In this subsection, we consider replacing Gauss randomized matrix with a sketch of input matrix $A$, i.e., choosing $\Omega_{2}=Y_{1}^{T}=(A\Omega_{1})^{T}$. As in \cite{Kaloorazi2018subspace-orbit}, we propose a subspace-orbit single-pass randomized QLP decomposition algorithm (SORQLP) for computing a low-rank approximation of $A$.

From the analysis in Subsection \ref{subsect:SPRQLP},  for the matrices $V$ and $B$ generated by Algorithm 2, we have $A\approx VB$. Premultiplying on both sides of  $A\approx VB$  by $Y_{1}^{T}$ yields
 $$Y_{2}=Y_{1}^{T}A\approx Y_{1}^{T}VB.$$

Thus, $$B\approx(Y_{1}^{T}V)^{\dag}Y_{2}.$$

The SORQLP is described as in Algorithm 4.

\begin{algorithm}\label{algor:SORQLP}
\caption{Subspace-orbit single-pass randomized QLP decomposition (SORQLP)} % Ëã·šµÄÃû×Ö
\hspace*{0.02in} {\bf Input:} %Ëã·šµÄÊäÈë£¬ \hspace*{0.02in}ÓÃÀŽ¿ØÖÆÎ»ÖÃ£¬Í¬Ê±ÀûÓÃ \\ œøÐÐ»»ÐÐ
$A\in\mathbb{R}^{m\times n}$, target rank:  $k\geq2$, oversampling parameter: $p\geq2$, number of columns sampled: $l=k+p$.\\
\hspace*{0.02in} {\bf Output:} %Ëã·šµÄœá¹ûÊä³ö
Matrices $Q,L,P$ such that $A\approx QLP^{T}$, where $Q$ is a column orthogonal matrix, \\$P$ is a orthogonal matrix and $L$ is a lower triangular matrix.
\begin{algorithmic}[1]
\State $\Omega=\texttt{randn}(n,l)$. % \State ºóÐŽÒ»°ãÓïŸä
\State $Y_{1}=A\Omega$, $Y_{2}=Y_{1}^{T}A$.
\State $[V,R]=\texttt{qr}(Y_{1},0)$.
\State $B=(Y_{1}^{T}V)^{\dag}Y_{2}$.
\State $[\widehat{Q},L,P]=\textrm{QLP}(B)$.
\State $Q=V\widehat{Q}$.
\end{algorithmic}
\end{algorithm}

%\subsubsection{Computational complexity analysis}
We have the following remarks on the computational  complexity of Algorithm 4.

\begin{itemize}
\item{Step 1:} Generating random matrices $\Omega$ takes $\mathcal{O}(nl)$ operations;
\item{Step 2:} Computing $Y_{1}=A\Omega$ and $Y_{2}=Y_{1}^{T}A$ takes $\mathcal{O}(mnl)$ operations;
\item{Step 3:} Computing unpivoted QR decomposition of $Y_{1}$ of size $m\times l$, takes $\mathcal{O}(ml^{2})$ operations;
\item{Step 4:} Computing $Y_{1}^{T}V$ takes $\mathcal{O}(ml^{2})$ operations, computing Moore-Penrose inverse of $Y_{1}^{T}V$ takes $\mathcal{O}(l^{3})$ operations, and multiplying it by $Y_{2}$ takes $\mathcal{O}(nl^{2})$ operations;
\item{Step 5:} Computing column pivoting QR decomposition of $B$ of size $l\times n$, takes $\mathcal{O}(nl^{2})$ operations;
\item{Step 6:} Computing column pivoting QR decomposition of $R_{0}^{T}$ of size $n\times l$, takes $\mathcal{O}(nl^{2})$ operations;
\item{Step 7:} Computing $VQ_{0}P_{1}$ takes $\mathcal{O}(ml^{2})$ operations, computing $P_{0}Q_{1}$ takes $\mathcal{O}(n^{3})$ operations.
\end{itemize}

Therefore, the total complexity of Algorithm 4 is $$\mathcal{C}_{SORQLP}=\mathcal{O}(nl+mnl+(m+n)l^{2}+l^{3}+n^{3}),$$
which is the same order of magnitude as the complexity of Algorithm 2. Similarly, considering the cost of data communication, the computation cost of Algorithm 4 is much cheaper than Algorithm 2.

%--------------------------------------------------------
\section{Error analysis}
In this section, we evaluate the performance of the proposed two single-pass QLP decomposition algorithms in terms of matrix approximation error and singular value approximation error.

In what follows, we assume that all Gaussian random matrices have full rank. We first recall some necessary lemmas.

\begin{lem}\cite{Shabat2018RLU}\label{lem:AGF-F}
Let $A\in \mathbb{R}^{m\times n}$ and let $k$ and $l$ be positive integers such that $\left(1+\frac{1}{\ln k}\right)k<l<\min\{m,n\}$.  If $\Omega\in\mathbb{R}^{n\times l}$ is a standard Gaussian random matrix, then there exists a matrix $F\in\mathbb{R}^{l\times n}$ such that
\begin{equation}\label{equ:4.1}
\|A\Omega F-A\|_{2}\leq\sqrt{\frac{a_{1}^{2}n}{c_{1}^{2}l}+1}\cdot\sigma_{k+1}(A),
\end{equation}
and
\begin{equation}\label{equ:4.2}
\|F\|_{2}\leq\frac{1}{c_{1}\sqrt{l}}
\end{equation}
with probability not less than $1-e^{-c_{2}l}-e^{-a_{2}n}$, where  the constants $a_{1},a_{2},c_{1}$, and $c_{2}$ are defined  as in Lemma \ref{lem:smallest-bound}.
\end{lem}

\begin{lem}\cite{Guming2015Subspaceiter}\label{lem:proj-error}
Let $A\in\mathbb{R}^{m\times n}$, whose  SVD is given by
\BE\label{a:svd}
A=U_m\Sigma V_n^{T},
\EE
where $\Sigma=\diag(\sigma_1(A),\ldots,\sigma_q(A))\in\Rmn$, $U_m\in\R^{m\times m}$ and  $V_n\in\R^{n\times n}$ are two orthogonal matrices. Suppose  $k\leq l$, $0\leq p\leq l-k$, and $\Omega\in\mathbb{R}^{n\times l}$ is a standard Gaussian random matrix. Let the reduced QR decomposition of $A\Omega$
\BE\label{a:qr}
A\Omega=QR,
\EE
where $Q\in\R^{m\times l}$ has orthonormal  columns and $R\in\R^{l\times l}$ is  upper triangular. Let \BE\label{p:vo}
V_n^{T}\Omega=\left[
                                           \begin{array}{c}
                                             \hat{\Omega}_{1} \\
                                             \hat{\Omega}_{2} \\
                                           \end{array}
                                         \right],\quad \hat{\Omega}_{1}\in\mathbb{R}^{k\times l}.
\EE
If $\hat{\Omega}_{1}$ has full row rank, then
\begin{equation}
\|A-QQ^{T}A\|_{F}\leq\sqrt{\frac{k\sigma_{1}^{2}(A)\sigma_{l-p+1}^{2}(A)\|\hat{\Omega}_{2}\|_{2}^{2}\|\hat{\Omega}_{1}^{\dag}\|_{2}^{2}}{\sigma_{l-p+1}^{2}(A)\|\hat{\Omega}_{2}\|_{2}^{2}\|\hat{\Omega}_{1}^{\dag}\|_{2}^{2}+\sigma_{1}^{2}(A)}+\sum\limits_{i=k+1}^{n}\sigma_{i}^{2}(A)},
\end{equation}
\begin{equation}\label{equ:4.4}
\|A-QQ^{T}A\|_{2}\leq\sqrt{\frac{k\sigma_{1}^{2}(A)\sigma_{l-p+1}^{2}(A)\|\hat{\Omega}_{2}\|_{2}^{2}\|\hat{\Omega}_{1}^{\dag}\|_{2}^{2}}{\sigma_{l-p+1}^{2}(A)\|\hat{\Omega}_{2}\|_{2}^{2}\|\hat{\Omega}_{1}^{\dag}\|_{2}^{2}+\sigma_{1}^{2}(A)}+\sigma_{k+1}^{2}(A)}.
\end{equation}
\end{lem}

%Lemma \ref{lem:bound-Omega} appears in Theorem 5.8 of \cite{Guming2015Subspaceiter}, it presents upper bound of $\|\hat{\Omega}_{2}\|_{2}^{2}\|\hat{\Omega}_{1}^{\dag}\|_{2}^{2}$.

\begin{lem}\cite{Guming2015Subspaceiter}\label{lem:bound-Omega}
Under the same assumptions of Lemma \ref{lem:proj-error}, for any $0<\Delta\ll1$,  we have
\begin{equation*}
\|\hat{\Omega}_{2}\|_{2}\|\hat{\Omega}_{1}^{\dag}\|_{2}\leq\mathcal{C}_{\Delta}
\end{equation*}
with probability not less than $1-\Delta$, where $\mathcal{C}_{\Delta}=\frac{e\sqrt{l}}{p+1}\left(\frac{2}{\Delta}\right)^{\frac{1}{p+1}}\left(\sqrt{n-l+p}+\sqrt{l}+\sqrt{2\log\frac{2}{\Delta}}\right)$.
\end{lem}

The following lemmas provide some inequalities about singular value, which are helpful to prove the main theorem related to the singular value approximation error of the proposed algorithms.

\begin{lem}(\cite[Theorem 3.3.16]{Horn1991Topics})\label{lem:topic-Matrix}
Let $G,\Delta G\in\mathbb{R}^{m\times n}$. Then  we have
\begin{equation*}
|\sigma_{i}(G+\Delta G)-\sigma_{i}(G)|\leq\sigma_{1}(\Delta G)
\end{equation*}
for $i=1,\ldots,q=\min\{m,n\}$.
\end{lem}

\begin{lem}\cite{Goulb2013MC}\label{lem:singular-Goulb}
Let $A\in\mathbb{R}^{m\times n}$. Suppose $Q$ is a matrix with orthonormal columns.  Then
\begin{equation*}
\sigma_{j}(A)\geq\sigma_{j}(Q^{T}A)
\end{equation*}
for $j=1,\ldots, q=\min\{m,n\}$.
\end{lem}

\begin{lem}\cite{Guming2015Subspaceiter}\label{lem:singular-guming}
Let $A\in\mathbb{R}^{m\times n}$ admit the  SVD  as in \eqref{a:svd}.  Let $\Omega\in\mathbb{R}^{n\times l}$ be a standard Gaussian matrix. Let $p$ be such that $0\leq p\leq l-k$ and partition $V_n^{T}\Omega=\left[
             \begin{array}{c}
               \hat{\Omega}_{1} \\
               \hat{\Omega}_{2} \\
             \end{array}
             \right]$.
Let $B=Q^{T}A$ with $Q$ being a matrix with orthonormal columns. If the matrix $\hat{\Omega}_{1}$ has full row rank, then
\begin{equation*}
\sigma_{j}(Q^{T}A)=\sigma_{j}(B)\geq\frac{\sigma_{j}(A)}{\sqrt{1+\|\hat{\Omega}_{2}\|_{2}^{2}\|\hat{\Omega_{1}^{\dag}}\|_{2}^{2}\left(\frac{\sigma_{l-p+1}(A)}{\sigma_{j}(A)}\right)^{2}}}
\end{equation*}
for $j=1,\ldots,q=\min\{m,n\}$.
\end{lem}

\subsection{Error analysis of Algorithm 3}
\subsubsection{Matrix approximation error analysis}
The following theorem provides some bounds for the matrix approximation error of Algorithm 3 in the sense of $2$-norm and Frobenius norm.

\begin{thm}\label{thm:matrix-error1}
Let $A\in\mathbb{R}^{m\times n}$ and let $k$ be  target rank. Suppose  $l_{1}$ and $l_{2}$ are such that $l_{1}>(1+\frac{1}{\ln k})k$ and $l_{2}>(1+\frac{ 1}{\ln l_{1}})l_{1}$.  Let $Q$, $L$, and $P$ be generated by Algorithm 3. Then, for any $0<\Delta\ll1$, we have
\begin{equation}\label{equ:4.8}
\|A-QLP^{T}\|_{2}\leq2(1+\frac{a_{1}\sqrt{m}}{c_{1}\sqrt{l_{2}}})\sqrt{\frac{a_{1}^{2}n}{c_{1}^{2}l_{1}}+1}\cdot\sigma_{k+1}(A)
\end{equation}
with probability not less than $1-e^{-a_{2}m}-e^{-a_{2}n}-e^{-c_{2}l_{1}}-e^{-c_{2}l_{2}}$ and
\begin{equation}\label{equ:4.9}
\|A-QLP\|_{F}\leq\left(1+\frac{a_{1}\sqrt{m}}{c_{1}\sqrt{l_{2}}}\right)\sqrt{\frac{k\sigma_{1}^{2}(A)\sigma_{k+1}^{2}(A)\mathcal{C}_{\Delta}^{2}}{\sigma_{k+1}^{2}(A)\mathcal{C}_{\Delta}^{2}+\sigma_{1}^{2}(A)}+\sum\limits_{i=k+1}^{n}\sigma_{i}^{2}(A)}
\end{equation}
with probability not less than $1-e^{-a_{2}m}-e^{-c_{2}l_{2}}-\Delta$, where
\[\mathcal{C}_{\Delta}=\frac{e\sqrt{l_{1}}}{p+1}\left(\frac{2}{\Delta}\right)^{\frac{1}{p+1}}\left(\sqrt{n-k}+\sqrt{l_{1}}+\sqrt{2\log\frac{2}{\Delta}}\right).
\]
Here, $a_{1}, c_{1}$ and $c_{2}$ are defined as in Definition \ref{def:general-gaussian} and Remark \ref{rem:parameter0}--\ref{rem:parameter}.
\end{thm}

\begin{proof}
Firstly, we consider matrix approximation error in the $2$-norm. From Algorithm 3 we have
\begin{eqnarray} \label{equ:4.10}
\|A-QLP^{T}\|_{2}&=&\|A-VB\|_{2}=\|A-V(\Omega_{2}V)^{\dag}\Omega_{2}A\|_{2} \nonumber\\
&\leq&\|A-VV^{T}A\|_{2}+\|VV^{T}A-V(\Omega_{2}V)^{\dag}\Omega_{2}A\|_{2}.
\end{eqnarray}
The second term  of the right hand side of \eqref{equ:4.10} is reduced to
\begin{eqnarray}
\|VV^{T}A-V(\Omega_{2}V)^{\dag}\Omega_{2}A\|_{2}&=&\|V(\Omega_{2}V)^{\dag}\Omega_{2}VV^{T}A-V(\Omega_{2}V)^{\dag}\Omega_{2}A\|_{2}\nonumber\\
&\leq&\|V(\Omega_{2}V)^{\dag}\Omega_{2}\|_{2}\|A-VV^{T}A\|_{2},
\end{eqnarray}
where the first equality follows from the fact that $\Omega_{2}V$ has full column rank and thus $(\Omega_{2}V)^{\dag}\Omega_{2}V=I$ since $\Omega_{2}V\in\mathbb{R}^{l_{2}\times l_{1}}$$(l_{2}\geq l_{1})$ is a Gaussian random matrix. Thus,
\begin{equation}\label{equ:4.11}
\|A-QLP^{T}\|_{2}\leq(1+\|V(\Omega_{2}V)^{\dag}\Omega_{2}\|_{2})\|A-VV^{T}A\|_{2}.
\end{equation}
For the first term of the right hand side of \eqref{equ:4.10}, by using the triangular inequality, there exists a matrix $F\in\mathbb{R}^{l_{1}\times n}$ such that
\begin{equation}\label{equ:4.12}
\|A-VV^{T}A\|_{2}\leq\|VV^{T}A-VV^{T}A\Omega_{1}F\|_{2}+\|VV^{T}A\Omega_{1}F-A\Omega_{1}F\|_{2}+\|A\Omega_{1}F-A\|_{2}.
\end{equation}
For the first term of the right hand side of  (\ref{equ:4.12}), we have
\BE\label{412:1t}
\|VV^{T}A-VV^{T}A\Omega_{1}F\|_{2}\leq\|VV^{T}\|_{2}\|A\Omega_{1}F-A\|_{2}=\|A\Omega_{1}F-A\|_{2}.
\EE
For the second term of the right hand side of  (\ref{equ:4.12}), we have
\BE\label{412:2t}
\|VV^{T}A\Omega_{1}F-A\Omega_{1}F\|_{2}\leq\|VV^{T}A\Omega_{1}-A\Omega_{1}\|_{2}\|F\|_{2}
=\|VV^{T}VR-VR\|_{2}\|F\|_{2}=0.
\EE
By substituting \eqref{412:1t} and \eqref{412:2t} into (\ref{equ:4.12}), we obtain $$\|A-VV^{T}A\|_{2}\leq2\|A\Omega_{1}F-F\|_{2}.$$
By hypothesis, $l_{1}>\left(1+\frac{1}{\ln k}\right)k$. Using Lemma \ref{lem:AGF-F} we have
\begin{equation}\label{equ:error1}
\|A-VV^{T}A\|_{2}\leq2\sqrt{\frac{a_{1}^{2}n}{c_{1}^{2}l_{1}}+1}\cdot\sigma_{k+1}(A)
\end{equation}
with probability not less than $1-e^{-c_{2}l_{1}}-e^{-a_{2}n}$. Furthermore,
\BE\label{eq:vo2}
\|V(\Omega_{2}V)^{\dag}\Omega_{2}\|_{2}=\|(\Omega_{2}V)^{\dag}\Omega_{2}\|_{2}\leq\|(\Omega_{2}V)^{\dag}\|_{2}\|\Omega_{2}\|_{2}.
\EE
We already know that $\Omega_{2}V\in\mathbb{R}^{l_{2}\times l_{1}}$ is a Gaussian random matrix.
By hypothesis, $l_{2}>\left(1+\frac{1}{\ln l_{1}}\right)l_{1}$.  Thus, by Lemma \ref{lem:smallest-bound}  we have
\begin{equation}\label{equ:error2}
\|(\Omega_{2}V)^{\dag}\|_{2}=\frac{1}{\sigma_{l_{1}}(\Omega_{2}V)}\leq\frac{1}{c_{1}\sqrt{l_{2}}}
\end{equation}
with probability not less than $1-e^{-c_{2}l_{2}}$. By Definition \ref{def:general-gaussian}, we have
\begin{equation}\label{equ:error3}
\|\Omega_{2}\|_{2}\leq a_{1}\sqrt{m}
\end{equation}
with probability not less than $1-e^{-a_{2}m}$. Substituting (\ref{equ:error2}) and (\ref{equ:error3}) into
\eqref{eq:vo2} yields
\BE\label{eq:vo2-bd}
\|V(\Omega_{2}V)^{\dag}\Omega_{2}\|_{2}\leq \frac{a_{1}\sqrt{m}}{c_{1}\sqrt{l_{2}}}
\EE
with probability not less than $1-e^{-a_{2}m}-e^{-c_{2}l_{2}}$.
Plugging (\ref{equ:error1}) and \eqref{eq:vo2-bd} into (\ref{equ:4.11}) gives rise to
\begin{equation*}
\|A-QLP^{T}\|_{2}\leq2\left(1+\frac{a_{1}\sqrt{m}}{c_{1}\sqrt{l_{2}}}\right)\sqrt{\frac{a_{1}^{2}n}{c_{1}^{2}l_{1}}+1}\cdot\sigma_{k+1}(A)
\end{equation*}
with probability not less than $1-e^{-a_{2}m}-e^{-a_{2}n}-e^{-c_{2}l_{1}}-e^{-c_{2}l_{2}}$.

Next, we consider matrix approximation error in the Frobenius norm. By using the similar  error analysis in the $2$-norm we have
\begin{eqnarray}\label{eq:vao2}
\|A-QLP^{T}\|_{F}&=&\|A-VB\|_{F}=\|A-V(\Omega_{2}V)^{\dag}\Omega_{2}A\|_{F}\nonumber\\
&\leq&\|A-VV^{T}A\|_{F}+\|VV^{T}A-V(\Omega_{2}V)^{\dag}\Omega_{2}A\|_{F}.
\end{eqnarray}
The second term of the right hand side of \eqref{eq:vao2} is reduced to
\begin{eqnarray*}
\|VV^{T}A-V(\Omega_{2}V)^{\dag}\Omega_{2}A\|_{F}&=&\|V(\Omega_{2}V)^{\dag}\Omega_{2}VV^{T}A-V(\Omega_{2}V)^{\dag}\Omega_{2}A\|_{F}\nonumber\\
&\leq&\|V(\Omega_{2}V)^{\dag}\Omega_{2}\|_{2}\|A-VV^{T}A\|_{F},
\end{eqnarray*}
where the first equality follows from the fact that the Gaussian matrix $\Omega_{2}V$ has full column rank and thus $(\Omega_{2}V)^{\dag}\Omega_{2}V=I$. Thus,
\begin{equation}
\|A-QLP^{T}\|_{F}\leq (1+\|V(\Omega_{2}V)^{\dag}\Omega_{2}\|_{2}) \|A-VV^{T}A\|_{F}.
\end{equation}
Let $V_n$ be defined by \eqref{a:svd} and
\BE\label{p:vo}
V_n^{T}\Omega_1=\left[
                                           \begin{array}{c}
                                             \hat{\Omega}_{1} \\
                                             \hat{\Omega}_{2} \\
                                           \end{array}
                                         \right],\quad \hat{\Omega}_{1}\in\mathbb{R}^{k\times l_1}.
\EE
By Lemma \ref{lem:proj-error}, we get
\begin{equation}\label{equ:4.18}
\|A-VV^{T}A\|_{F}\leq\sqrt{\frac{k\sigma_{1}^{2}(A)\sigma_{k+1}^{2}(A)\|\hat{\Omega}_{2}\|_{2}^{2}\|\hat{\Omega}_{1}^{\dag}\|_{2}^{2}}{\sigma_{k+1}^{2}(A)\|\hat{\Omega}_{2}\|_{2}^{2}\|\hat{\Omega}_{1}^{\dag}\|_{2}^{2}+\sigma_{1}^{2}(A)}+\sum\limits_{i=k+1}^{n}\sigma_{i}^{2}(A)}.
\end{equation}
From Lemma \ref{lem:bound-Omega} and  \eqref{eq:vo2-bd} we obtain, for any $0<\Delta\ll1$,
\begin{equation}\label{equ:qlp-est}
\|A-QLP^{T}\|_{F}\leq\left(1+\frac{a_{1}\sqrt{m}}{c_{1}\sqrt{l_{2}}}\right)\sqrt{\frac{k\sigma_{1}^{2}(A)\sigma_{k+1}^{2}(A)\mathcal{C}_{\Delta}^{2}}{\sigma_{k+1}^{2}(A)\mathcal{C}_{\Delta}^{2}+\sigma_{1}^{2}(A)}+\sum\limits_{i=k+1}^{n}\sigma_{i}^{2}(A)}.
\end{equation}
with probability not less than $1-e^{-a_{2}m}-e^{-c_{2}l_{2}}-\Delta$.
\end{proof}

\begin{rem}
In  Theorem \ref{thm:matrix-error1}, we require that $l_{1}>\left(1+\frac{1}{\ln k}\right)k$ and $l_{2}>\left(1+\frac{1}{\ln l_{1}}\right)l_{1}$. When $l_2\ge l_1\ge k$ with $l_1\approx k$ and  $l_2\approx l_1$, the similar bounds of matrix approximation error can be derived by using the results in \cite{Litvak2010smallest}.
\end{rem}

\begin{rem}\label{rem:bd-pro}
Using   \eqref{equ:4.4} in Lemma \ref{lem:proj-error}, Lemma \ref{lem:bound-Omega}, \eqref{equ:4.11}, and (\ref{eq:vo2-bd}), for $Q$, $L$, and $P$ generated by Algorithm 3, we have
\begin{equation}
\|A-QLP^{T}\|_{2}\leq\left(1+\frac{a_{1}\sqrt{m}}{c_{1}\sqrt{l_{2}}}\right)
\sqrt{\frac{k\sigma_{1}^{2}(A)\sigma_{k+1}^{2}(A)\mathcal{C}_{\Delta}^{2}}{\sigma_{k+1}^{2}(A)
\mathcal{C}_{\Delta}^{2}+\sigma_{1}^{2}(A)}+\sigma_{k+1}^{2}(A)}
\end{equation}
with probability not less than $1-e^{-a_{2}m}-e^{-c_{2}l_{2}}-\Delta$.
\end{rem}

\begin{rem}
In fact, in  Theorem \ref{thm:matrix-error1}, it is necessary to assume that the test matrix $\Omega_2$ has full row rank  as in  \cite[Theorem 9.1]{Halko2011SIAMreview}. In this case, we have $(\Omega_{2}V)^{\dag}\Omega_{2}V=I$ since  $\Omega_2^\dag\Omega_2=I$.
%Next, we show that the probability of full rank of Gaussian matrix is 1.
%Let $A$ is non-full rank, and $A=U_m\Sigma V_n^{T}$ is SVD of $A$, where $\Sigma=\diag[\sigma_{1}(A), \ldots, \sigma_{k}(A), 0, \ldots, 0]$, $U_m\in\mathbb{R}^{m\times m}$, $V_n\in\mathbb{R}^{n\times n}$ are two orthogonal matrices. Define a matrix sequence $\{A_{k}\}$, $A_{k}=U_m\widehat{\Sigma}(k) V_n^{T}$, where $\widehat{\Sigma}(k)=\diag[\sigma_{1}(A), \ldots, \sigma_{k}(A), \frac{1}{k}, \ldots, \frac{1}{k}]$. We can assert that there are infinitely many full rank matrices in any neighborhood of matrix $A$, thus the set of full-rank matrices is a dense subset of $\mathbb{R}^{m\times n}$. Then the probability of non-full rank of Gaussian random matrices is 0, and the probability of full rank is 1.
\end{rem}

\subsubsection{Singular value approximation error analysis} \label{sec412}
We first give some lemmas on the error bounds for singular values of the QLP decomposition for a matrix $A$.
As noted in \cite{Stewart1999QLP}, for the QLP decomposition, the pivoting in the first QR decomposition is essential while the  pivoting of the second QR decomposition is  only necessary to avoid ``certain contrived counterexamples".  Therefore, in order to simplify the analysis, we assume that there is no pivoting in the second QR decomposition of the QLP decomposition for the matrix $B$ generated by Algorithm 3.

The following lemma gives a bound for the maximum singular value approximation error.

\begin{lem}\cite{Huckaby2003QLP}\label{lem:l11-smallsingular}
Let $A\in\mathbb{R}^{m\times n}$ and $\sigma_{1}(A)>\sigma_{2}(A)$. Let $R_0$ be the $R$-factor in the pivoted QR factorization of $A$, $AP_{0}=Q_{0}R_{0}$ and let $L^T$ be the  $R$-factor in the unpivoted QR factorization of $(R_0)^T$, $R_{0}^{T}=Q_{1}L^{T}$. Partition $R_{0}$ and $L$ as
\[
R_{0}=\left[
                                                                                    \begin{array}{cc}
                                                                                      r_{11} & R_{12} \\
                                                                                      0 & R_{22} \\
                                                                                    \end{array}
                                                                                  \right]\quad\mbox{and}\quad L=\left[
                                                                                             \begin{array}{cc}
                                                                                               l_{11} & 0 \\
                                                                                               L_{21} & L_{22} \\
                                                                                             \end{array}
                                                                                           \right],
\]
where $r_{11}$, $l_{11}\in\R$. If $\|R_{22}\|_{2}\leq\sqrt{2(q-1)}\sigma_{2}(A)$ and $\frac{\|R_{22}\|_{2}}{|r_{11}|}<1$, then
\begin{equation*}
|l_{11}|^{-1}-\sigma_{1}^{-1}(A)\leq\frac{\sigma_{2}^{2}(A)}{\sigma_{1}^{3}(A)}\mathcal{O}\left(\frac{q^{\frac{5}{2}}\|R_{12}\|_{2}^{2}}{(1-\rho_{1}^{2})l_{11}^{2}}\right),
\end{equation*}
where $q=\min\{m,n\}$ and $\rho_1=\frac{\|L_{22}\|_{2}}{|l_{11}|}$.
\end{lem}

The following lemma  provides   some bounds for the interior singular value approximation errors.

\begin{lem}\cite{Huckaby2003QLP}\label{lem:block-error}
Let $A\in\mathbb{R}^{m\times n}$ and $\sigma_{k}(A)>\sigma_{k+1}(A)$. Let $R_0$ be the $R$-factor in the pivoted QR factorization of $A$, $AP_{0}=Q_{0}R_{0}$ and let $L^T$ be the  $R$-factor in the unpivoted QR factorization of $(R_0)^T$, $R_{0}^{T}=Q_{1}L^{T}$. Partition $R_{0}$ and $L$ as
\[
R_{0}=\left[
                                                             \begin{array}{cc}
                                                               R_{11} & R_{12} \\
                                                               0 & R_{22} \\
                                                             \end{array}
                                                             \right], L=\left[
                                                                        \begin{array}{cc}
                                                                          L_{11} & 0 \\
                                                                          L_{21} & L_{22} \\
                                                                        \end{array}
                                                                        \right],
\]
where $R_{11}$, $L_{11}\in\mathbb{R}^{k\times k}$. If $\|R_{22}\|_{2}\leq\sqrt{(k+1)(q-k)}\sigma_{k+1}(A)$, $\sigma_{k}(R_{11})\geq\frac{\sigma_{k}(A)}{\sqrt{k(q-k+1)}}$, and $\frac{\|R_{22}\|_{2}}{\sigma_{k}(R_{11})}<1$, then for $i=1,\ldots,n-k$,
\begin{equation}\label{equ:4.21}
\frac{\sigma_{j}(L_{22})-\sigma_{k+j}(A)}{\sigma_{k+j}(A)}\leq\left(\frac{\sigma_{k+1}(A)}{\sigma_{k}(A)}\right)^{2}\mathcal{O}\left(\frac{q^{\frac{5}{2}}\|R_{12}\|_{2}^{2}}{(1-\rho_{1}^{2})\sigma_{k}^{2}(R_{11})}\right)
\end{equation}
and for $j=1,\ldots,k$,
\begin{equation}\label{equ:4.22}
\frac{\sigma_{j}^{-1}(L_{11})-\sigma_{j}^{-1}(A)}{\sigma_{j}^{-1}(A)}\leq\left(\frac{\sigma_{k+1}(A)}{\sigma_{k}(A)}\right)^{2}\mathcal{O}\left(\frac{q^{\frac{5}{2}}\|R_{12}\|_{2}^{2}}{\left(1-\rho_{1}^{2}\right)\sigma_{k}^{2}(R_{11})}\right),
\end{equation}
where  $q=\min\{m,n\}$ and $\rho_{1}=\frac{\|L_{22}\|_{2}}{\sigma_{k}(L_{11})}$.
\end{lem}

Now we are ready to state our first theorem on singular value approximation error of  Algorithm 3.

Let $B$ be generated by  Algorithm 3. Then we can rewrite $B$  as
\[
B=(\Omega_{2}V)^{\dag}\Omega_{2}A-V^{T}A+V^{T}A=B-V^{T}A+V^{T}A.
\]
It follows from Lemma \ref{lem:topic-Matrix} that for any $1\leq j\leq l_1$, $$|\sigma_{j}(B)-\sigma_{j}(V^{T}A)|\leq\sigma_{1}(B-V^{T}A).$$ Thus, for any $1\leq j\leq l_1$,
\begin{equation}\label{equ:4.23}
\sigma_{j}(V^{T}A)-\sigma_{1}(B-V^{T}A)\leq\sigma_{j}(B)\leq\sigma_{j}(V^{T}A)+\sigma_{1}(B-V^{T}A).
\end{equation}
It is easy to see that $\sigma_{j}(B)=\sigma_{j}(L)$ for all $j=1,\ldots,l_1$. Thus, for any $1\leq j\leq l_1$,
\begin{equation}\label{equ:4.24}
\sigma_{j}(V^{T}A)-\sigma_{1}(B-V^{T}A)\leq\sigma_{j}(L)\leq\sigma_{j}(V^{T}A)+\sigma_{1}(B-V^{T}A).
\end{equation}

\begin{thm}\label{thm:singular-error1}
Let $A\approx QLP^{T}$ be rank-$k$ SPRQLP decomposition produced by Algorithm 3. For $B$ generated by  Algorithm 3, let $R_0$ be the $R$-factor in the pivoted QR factorization of $B$, $BP_{0}=Q_{0}R_{0}$ and let $L^T$ be the  $R$-factor in the unpivoted QR factorization of $(R_0)^T$, $R_{0}^{T}=Q_{1}L^{T}$.
Then, for any $0<\Delta\ll1$, we have $$\sigma_{j}(L)\leq\mathcal{C}+\sigma_{j}(A)$$ with probability not less than $1-e^{-a_{2}m}-e^{-a_{2}n}-e^{-c_{2}l_{1}}-e^{-c_{2}l_{2}}$ for all $j=1,\dots, k$ and $$\sigma_{j}(L)\geq\frac{\sigma_{j}(A)}{\rho}-\mathcal{C}$$ with probability not less than $1-e^{-a_{2}m}-e^{-a_{2}n}-e^{-c_{2}l_{1}}-e^{-c_{2}l_{2}}-\Delta$ for all $j=1,\dots, k$, where
\[
\mathcal{C}=2\frac{a_{1}\sqrt{m}}{c_{1}\sqrt{l_{2}}}\sqrt{\frac{a_{1}^{2}n}{c_{1}^{2}l_{1}}+1}\cdot\sigma_{k+1}(A),\quad
\rho=\sqrt{1+\mathcal{C}_{\Delta}^{2}\left(\frac{\sigma_{k+1}(A)}{\sigma_{j}(A)}\right)^{2}}
\]
with
\[
\mathcal{C}_{\Delta}=\frac{e\sqrt{l_{1}}}{p+1}\left(\frac{2}{\Delta}\right)^{\frac{1}{p+1}}
\left(\sqrt{n-k}+\sqrt{l_{1}}+\sqrt{2\log\frac{2}{\Delta}}\right),
\]
and the parameters $a_{1}, c_{1}$ and $c_{2}$ are defined as in Definition \ref{def:general-gaussian} and Remark \ref{rem:parameter0}--\ref{rem:parameter}.
\end{thm}

\begin{proof}
By Lemma \ref{lem:singular-Goulb} we obtain
\BE\label{eq:va-s}
\sigma_{j}(V^{T}A)\leq\sigma_{j}(A),\quad \forall 1\leq j\leq k.
\EE
Using  (\ref{equ:error1}) and \eqref{eq:vo2-bd} we have
\begin{eqnarray}\label{equ:4.25}
\sigma_{1}(B-V^{T}A)&=&\|(\Omega_{2}V)^{\dag}\Omega_{2}A-(\Omega_{2}V)^{\dag}\Omega_{2}VV^{T}A\|_{2}\nonumber\\
&\leq&\|(\Omega_{2}V)^{\dag}\Omega_{2}\|_{2}\|A-VV^{T}A\|_{2}\nonumber\\
&\leq&2\frac{a_{1}\sqrt{m}}{c_{1}\sqrt{l_{2}}}\sqrt{\frac{a_{1}^{2}n}{c_{1}^{2}l_{1}}+1}\cdot\sigma_{k+1}(A):=\mathcal{C}
\end{eqnarray}
with probability not less than $1-e^{-a_{2}m}-e^{-a_{2}n}-e^{-c_{2}l_{1}}-e^{-c_{2}l_{2}}$.  From  \eqref{equ:4.24}, \eqref{eq:va-s}, and \eqref{equ:4.25} we have
\[
\sigma_{j}(L)\leq\sigma_{j}(V^{T}A)+\sigma_{1}(B-V^{T}A)
\leq\mathcal{C}+\sigma_{j}(A)
\]
with probability not less than $1-e^{-a_{2}m}-e^{-a_{2}n}-e^{-c_{2}l_{1}}-e^{-c_{2}l_{2}}$ for all $j=1,\ldots, k$.

On the other hand, let $\hat{B}=V^{T}A$ and let the SVD of $A$ be given by \eqref{a:svd}. Denote
\[
V_n^T\Omega_{1}=\left[
                     \begin{array}{c}
                       \hat{\Omega}_{1} \\
                       \hat{\Omega}_{2} \\
                     \end{array}
                   \right],
\]
where $\hat{\Omega}_{1}\in\mathbb{R}^{k\times l_{1}}$, $\hat{\Omega}_{2}\in\mathbb{R}^{(n-k)\times l_{1}}$. According to Lemma \ref{lem:singular-guming}, we obtain $$\sigma_{j}(V^{T}A)=\sigma_{j}(\hat{B})=\sigma_{j}(\hat{B}_{k})\geq
\frac{\sigma_{j}(A)}{\sqrt{1+\|\hat{\Omega}_{2}\|_{2}^{2}\|\hat{\Omega}_{1}^{\dag}\|_{2}^{2}\left(\frac{\sigma_{k+1}(A)}
{\sigma_{j}(A)}\right)^{2}}},$$
for all $j=1,\ldots, k$. By Lemma \ref{lem:bound-Omega} we have, for any $0<\Delta\ll1$, $\|\hat{\Omega}_{2}\|_{2}^{2}\|\hat{\Omega}_{1}^{\dag}\|_{2}^{2}\leq\mathcal{C}_{\Delta}^{2}$ with probability not less than $1-\Delta$.
Then we have $\sigma_{j}(V^{T}A)\geq\frac{\sigma_{j}(A)}{\rho}$ with probability not less than $1-\Delta$. This, together with \eqref{equ:4.24}  and (\ref{equ:4.25}), yields
\[
\sigma_{j}(L)\geq\frac{\sigma_{j}(A)}{\rho}-\mathcal{C}
\]
with probability not less than $1-e^{-a_{2}m}-e^{-a_{2}n}-e^{-c_{2}l_{1}}-e^{-c_{2}l_{2}}-\Delta$  for all $j=1,\ldots, k$.
\end{proof}

We give the following theorem on the maximum singular value approximation
error of Algorithm 3.

\begin{thm}\label{thm:singular-error2}
Let $A\approx QLP^{T}$ be rank-$k$ SPRQLP decomposition produced by Algorithm 3.  For $B$ generated by  Algorithm 3, let $R_0$ be the $R$-factor in the pivoted QR factorization of $B$, $BP_{0}=Q_{0}R_{0}$ and let $L^T$ be the  $R$-factor in the unpivoted QR factorization of $(R_0)^T$, $R_{0}^{T}=Q_{1}L^{T}$.  Partition $R_{0}$ and $L$ as
\[
R_{0}=\left[
                                                                                    \begin{array}{cc}
                                                                                      r_{11} & R_{12} \\
                                                                                      0 & R_{22} \\
                                                                                    \end{array}
                                                                                  \right]\quad\mbox{and}\quad L=\left[
                                                                                             \begin{array}{cc}
                                                                                               l_{11} & 0 \\
                                                                                               L_{21} & L_{22} \\
                                                                                             \end{array}
                                                                                           \right],
\]
where $r_{11}$, $l_{11}\in\mathbb{R}$. Suppose $\sigma_{1}(B)>\sigma_{2}(B)$, $\|R_{22}\|_{2}\leq\sqrt{2(l_{1}-1)}\sigma_{2}(B)$, and $\frac{\|R_{22}\|_{2}}{|r_{11}|}<1$. Then for any given $0<\Delta\ll 1$,  we have
\begin{equation}
|l_{11}|^{-1}-\sigma_{1}^{-1}(A)\leq\sigma_{1}^{-1}(B)\left(1-\frac{1}{\rho}\cdot\frac{1}{1+\mathcal{C}\sigma_{1}^{-1}(B)}\right)+\frac{\sigma_{2}^{2}(B)}{\sigma_{1}^{3}(B)}\mathcal{O}\left(\frac{q^{\frac{5}{2}}\|R_{12}\|_{2}^{2}}{\left(1-\rho_{1}^{2}\right)l_{11}^{2}}\right)
\end{equation}
with probability not less than $1-e^{-a_{2}m}-e^{-a_{2}n}-e^{-c_{2}l_{1}}-e^{-c_{2}l_{2}}-\Delta$, where $\rho_{1}=\frac{\|L_{22}\|_{2}}{|l_{11}|}$,
\[
\mathcal{C}=2\frac{a_{1}\sqrt{m}}{c_{1}\sqrt{l_{2}}}\sqrt{\frac{a_{1}^{2}n}{c_{1}^{2}l_{1}}+1}\cdot\sigma_{k+1}(A)
\quad\mbox{and}\quad
\rho=\sqrt{1+\mathcal{C}_{\Delta}^{2}\left(\frac{\sigma_{k+1}(A)}{\sigma_{j}(A)}\right)^{2}},
\]
with
\[
\mathcal{C}_{\Delta}=\frac{e\sqrt{l_{1}}}{p+1}\left(\frac{2}{\Delta}\right)^{\frac{1}{p+1}}
\left(\sqrt{n-k}+\sqrt{l_{1}}+\sqrt{2\log\frac{2}{\Delta}}\right),
\]
and the parameters $a_{1}, c_{1}$ and $c_{2}$ are defined as in Definition \ref{def:general-gaussian} and Remark \ref{rem:parameter0}--\ref{rem:parameter}.
\end{thm}
\begin{proof}
By Theorem \ref{thm:singular-error1} we have
\begin{equation}\label{equ:4.27}
\sigma_{1}(B)=\sigma_{1}(L)\geq\frac{\sigma_{1}(A)}{\rho}-\mathcal{C},\quad\mbox{i.e.,}\quad \frac{1}{\sigma_{1}(A)}\geq\frac{1}{\rho\sigma_{1}(B)+\mathcal{C}\rho}
\end{equation}
with probability not less than $1-e^{-a_{2}m}-e^{-a_{2}n}-e^{-c_{2}l_{1}}-e^{-c_{2}l_{2}}-\Delta$.
By  Lemma \ref{lem:l11-smallsingular} we obtain $$|l_{11}|^{-1}\leq\sigma_{1}^{-1}(B)+\frac{\sigma_{2}^{2}(B)}{\sigma_{1}^{3}(B)}\mathcal{O}\left(\frac{q^{\frac{5}{2}}\|R_{12}\|_{2}^{2}}{\left(1-\rho_{1}^{2}\right)l_{11}^{2}}\right).$$
Subtracting $\sigma_{1}^{-1}(A)$ from both the sides of the above inequality gives rise to
\begin{eqnarray}
|l_{11}|^{-1}-\sigma_{1}^{-1}(A)&\leq&\sigma_{1}^{-1}(B)-\frac{1}{\sigma_{1}(A)}+\frac{\sigma_{2}^{2}(B)}{\sigma_{1}^{3}(B)}\mathcal{O}\left(\frac{q^{\frac{5}{2}}\|R_{12}\|_{2}^{2}}{\left(1-\rho_{1}^{2}\right)l_{11}^{2}}\right)\nonumber\\
&\leq&\sigma_{1}^{-1}(B)-\frac{1}{\rho\sigma_{1}(B)+\mathcal{C}\rho}+\frac{\sigma_{2}^{2}(B)}{\sigma_{1}^{3}(B)}\mathcal{O}\left(\frac{q^{\frac{5}{2}}\|R_{12}\|_{2}^{2}}{\left(1-\rho_{1}^{2}\right)l_{11}^{2}}\right)\nonumber\\
&=&\sigma_{1}^{-1}(B)\left(1-\frac{1}{\rho}\cdot\frac{1}{1+\mathcal{C}\sigma_{1}^{-1}(B)}\right)+\frac{\sigma_{2}^{2}(B)}{\sigma_{1}^{3}(B)}\mathcal{O}\left(\frac{q^{\frac{5}{2}}\|R_{12}\|_{2}^{2}}{\left(1-\rho_{1}^{2}\right)l_{11}^{2}}\right)\nonumber
\end{eqnarray}
with probability not less than $1-e^{-a_{2}m}-e^{-a_{2}n}-e^{-c_{2}l_{1}}-e^{-c_{2}l_{2}}-\Delta$, where the second inequality follows from (\ref{equ:4.27}).
\end{proof}

On  the interior singular value approximation errors of  Algorithm 3, we have the following theorem.
\begin{thm}\label{thm:singular-error3}
Let $A\approx QLP^{T}$ be rank-$k$ SPRQLP decomposition produced by Algorithm 3. For $B$ generated by  Algorithm 3, let $R_0$ be the $R$-factor in the pivoted QR factorization of $B$, $BP_{0}=Q_{0}R_{0}$ and let $L^T$ be the  $R$-factor in the unpivoted QR factorization of $(R_0)^T$, $R_{0}^{T}=Q_{1}L^{T}$. Suppose $\sigma_{s}(B)>\sigma_{s+1}(B)$ for some $1\leq s<k$. Partition $R_{0}$ and $L$ as
\[
R_{0}=\left[
                                   \begin{array}{cc}
                                     R_{11} & R_{12} \\
                                     0 & R_{22} \\
                                   \end{array}
                                 \right]\quad \mbox{and}\quad L=\left[
                                              \begin{array}{cc}
                                                L_{11} & 0 \\
                                                L_{21} & L_{22} \\
                                              \end{array}
                                            \right],
\]
where $R_{11}$, $L_{11}\in\mathbb{R}^{s\times s}$. Assume that $\|R_{22}\|_{2}\leq\sqrt{(s+1)(l_{1}-s)}\sigma_{s+1}(B)$, $\sigma_{s}(R_{11})\geq\frac{\sigma_{s}(B)}{\sqrt{s(l_{1}-s+1)}}$, and $\frac{\|R_{22}\|_{2}}{\sigma_s(R_{11})}<1$. Then for any given $0<\Delta\ll1$, we have
\begin{small}
\begin{equation}
\frac{\sigma_{j}^{-1}(L_{11})-\sigma_{j}^{-1}(A)}{\sigma_{j}^{-1}(A)}\leq\rho-1+\frac{\mathcal{C\rho}}{\sigma_{j}(B)}+\left(\frac{\mathcal{C\rho}}{\sigma_{j}(B)}+\rho\right)\left(\frac{\sigma_{s+1}(B)}{\sigma_{s}(B)}\right)^{2}\mathcal{O}\left(\frac{q^{\frac{5}{2}}\|R_{12}\|_{2}^{2}}{(1-\rho_{1}^{2})\sigma_{s}^2(L_{11})}\right)
\end{equation}
\end{small}
with probability not less than $1-e^{-a_{2}m}-e^{-a_{2}n}-e^{-c_{2}l_{1}}-e^{-c_{2}l_{2}}-\Delta$ for all $i=1,\ldots,s$;
\begin{small}
\begin{equation}
\frac{\sigma_{j}(L_{22})-\sigma_{s+j}(A)}{\sigma_{s+j}(A)}\leq\frac{\mathcal{C}}{\sigma_{s+j}(B)-\mathcal{C}}+\left(1+\frac{\mathcal{C}}{\sigma_{s+j}(B)-\mathcal{C}}\right)\left(\frac{\sigma_{s+1}(B)}{\sigma_{s}(B)}\right)^{2}\mathcal{O}\left(\frac{q^{\frac{5}{2}}\|R_{12}\|_{2}^{2}}{(1-\rho_{1}^{2})\sigma_{s}^2(L_{11})}\right)
\end{equation}
\end{small}
with probability not less than $1-e^{-a_{2}m}-e^{-a_{2}n}-e^{-c_{2}l_{1}}-e^{-c_{2}l_{2}}$ for all $j=1,\ldots,k-s$,  where $\rho_{1}=\frac{\|L_{22}\|_{2}}{\sigma_{s}(L_{11})}$,
\[
\mathcal{C}=2\frac{a_{1}\sqrt{m}}{c_{1}\sqrt{l_{2}}}\sqrt{\frac{a_{1}^{2}n}{c_{1}^{2}l_{1}}+1}\cdot\sigma_{k+1}(A)
\quad\mbox{and}\quad
\rho=\sqrt{1+\mathcal{C}_{\Delta}^{2}\left(\frac{\sigma_{k+1}(A)}{\sigma_{j}(A)}\right)^{2}},
\]
with
\[
\mathcal{C}_{\Delta}=\frac{e\sqrt{l_{1}}}{p+1}\left(\frac{2}{\Delta}\right)^{\frac{1}{p+1}}
\left(\sqrt{n-k}+\sqrt{l_{1}}+\sqrt{2\log\frac{2}{\Delta}}\right),
\]
and the parameters $a_{1}, c_{1}$ and $c_{2}$ are defined as in Definition \ref{def:general-gaussian} and Remark \ref{rem:parameter0}--\ref{rem:parameter}.
\end{thm}

\begin{proof}
From Theorem \ref{thm:singular-error1}, we know that
\begin{equation*}
\sigma_{j}(B)\leq\mathcal{C}+\sigma_{j}(A),\quad\mbox{i.e.,}\quad \frac{1}{\sigma_{j}(A)}\leq\frac{1}{\sigma_{j}(B)-\mathcal{C}}
\end{equation*}
with probability not less than $1-e^{-a_{2}m}-e^{-a_{2}n}-e^{-c_{2}l_{1}}-e^{-c_{2}l_{2}}$  for all $j=1,\ldots, k$; Similarly,
\begin{equation*}
\sigma_{j}(B)\geq\frac{\sigma_{j}(A)}{\rho}-\mathcal{C},\quad\mbox{i.e.,}\quad \sigma_{j}(A)\leq\mathcal{C}\rho+\rho\sigma_{j}(B)
\end{equation*}
with probability not less than $1-e^{-a_{2}m}-e^{-a_{2}n}-e^{-c_{2}l_{1}}-e^{-c_{2}l_{2}}-\Delta$  for all $j=1,\ldots, k$.

Using Lemma \ref{lem:block-error} we have, for $j=1,\ldots, k-s$,
\begin{small}
\[
\sigma_{j}(L_{22})-\sigma_{s+j}(A)\leq\sigma_{s+j}(B)-\sigma_{s+j}(A)+\frac{\sigma_{s+1}^{2}(B)\sigma_{s+j}(B)}
{\sigma_{s}^{2}(B)}\mathcal{O}\left(\frac{q^{\frac{5}{2}}\|R_{12}\|_{2}^{2}}{(1-\rho_{1}^{2})\sigma_{s}^2(L_{11})}\right).
\]
\end{small}
Dividing $\sigma_{s+j}(A)$ on both sides of the above inequality yields
\begin{small}
\begin{eqnarray}
\frac{\sigma_{j}(L_{22})-\sigma_{s+j}(A)}{\sigma_{s+j}(A)}&\leq&\frac{\sigma_{s+j}(B)-\sigma_{s+j}(A)}{\sigma_{s+j}(A)}+\frac{\sigma_{s+1}^{2}(B)\sigma_{s+j}(B)}{\sigma_{s+j}(A)\sigma_{s}^{2}(B)}\mathcal{O}\left(\frac{q^{\frac{5}{2}}\|R_{12}\|_{2}^{2}}{(1-\rho_{1}^{2})\sigma_{s}^2(L_{11})}\right)\nonumber\\
&\leq&\frac{\mathcal{C}}{\sigma_{s+j}(B)-\mathcal{C}}+\frac{\sigma_{s+1}^{2}(B)\sigma_{s+j}(B)}{(\sigma_{s+j}(B)-\mathcal{C})\sigma_{s}^{2}(B)}\mathcal{O}\left(\frac{q^{\frac{5}{2}}\|R_{12}\|_{2}^{2}}{(1-\rho_{1}^{2})\sigma_{s}^2(L_{11})}\right)\nonumber\\
&=&\frac{\mathcal{C}}{\sigma_{s+j}(B)-\mathcal{C}}+\left(1+\frac{\mathcal{C}}{\sigma_{s+j}(B)-\mathcal{C}}\right)\left(\frac{\sigma_{s+1}(B)}{\sigma_{s}(B)}\right)^{2}\mathcal{O}\left(\frac{q^{\frac{5}{2}}\|R_{12}\|_{2}^{2}}{(1-\rho_{1}^{2})\sigma_{s}^2(L_{11})}\right)\nonumber
\end{eqnarray}
\end{small}
with probability not less than $1-e^{-a_{2}m}-e^{-a_{2}n}-e^{-c_{2}l_{1}}-e^{-c_{2}l_{2}}$ for all $j=1,\ldots, k-s$.

By  Lemma \ref{lem:block-error} again we have, for $j=1,\ldots, s$, $$\sigma_{j}^{-1}(L_{11})-\sigma_{j}^{-1}(A)\leq\sigma_{j}^{-1}(B)-\sigma_{j}^{-1}(A)+\frac{\sigma_{s+1}^{2}(B)}{\sigma_{s}^{2}(B)\sigma_{j}(B)}\mathcal{O}\left(\frac{q^{\frac{5}{2}}\|R_{12}\|_{2}^{2}}{(1-\rho_{1}^{2})\sigma_{s}^2(L_{11})}\right).$$
Dividing $\sigma_{j}^{-1}(A)$ on both sides of the above inequality yields
\begin{small}
\begin{eqnarray}
&&\frac{\sigma_{j}^{-1}(L_{11})-\sigma_{j}^{-1}(A)}{\sigma_{j}^{-1}(A)}\\
&\leq&\frac{\sigma_{j}^{-1}(B)-\sigma_{j}^{-1}(A)}{\sigma_{j}^{-1}(A)}+\frac{\sigma_{j}(A)\sigma_{s+1}^{2}(B)}{\sigma_{s}^{2}(B)\sigma_{j}(B)}\mathcal{O}\left(\frac{q^{\frac{5}{2}}\|R_{12}\|_{2}^{2}}{(1-\rho_{1}^{2})\sigma_{s}^2(L_{11})}\right)\nonumber\\
&\leq&\frac{\sigma_{j}^{-1}(B)-(\rho\sigma_{j}(B)+\mathcal{C}\rho)^{-1}}{(\rho\sigma_{j}(B)+\mathcal{C}\rho)^{-1}}+\frac{(\rho\sigma_{j}(B)+\mathcal{C}\rho)\sigma_{s+1}^{2}(B)}{\sigma_{s}^{2}(B)\sigma_{j}(B)}\mathcal{O}\left(\frac{q^{\frac{5}{2}}\|R_{12}\|_{2}^{2}}{(1-\rho_{1}^{2})\sigma_{s}^2(L_{11})}\right)\nonumber\\
&=&\frac{\rho\sigma_{j}(B)+\mathcal{C}\rho}{\sigma_{j}(B)}-1+\frac{\rho\sigma_{j}(B)+\mathcal{C}\rho}{\sigma_{j}(B)}\left(\frac{\sigma_{s+1}(B)}{\sigma_{s}(B)}\right)^{2}\mathcal{O}\left(\frac{q^{\frac{5}{2}}\|R_{12}\|_{2}^{2}}{(1-\rho_{1}^{2})\sigma_{s}^2(L_{11})}\right)\nonumber\\
&=&\rho-1+\frac{\mathcal{C\rho}}{\sigma_{j}(B)}+\left(\frac{\mathcal{C\rho}}{\sigma_{j}(B)}+\rho\right)\left(\frac{\sigma_{s+1}(B)}{\sigma_{s}(B)}\right)^{2}\mathcal{O}\left(\frac{q^{\frac{5}{2}}\|R_{12}\|_{2}^{2}}{(1-\rho_{1}^{2})\sigma_{s}^2(L_{11})}\right)\nonumber
\end{eqnarray}
\end{small}
with probability not less than $1-e^{-a_{2}m}-e^{-a_{2}n}-e^{-c_{2}l_{1}}-e^{-c_{2}l_{2}}-\Delta$  for all $j=1,\ldots, s$.
\end{proof}

\begin{cor} Under the same assumptions of Theorem {\rm \ref{thm:singular-error3}}, if $s>k$, then
\begin{small}
\begin{equation*}
\begin{split}
\frac{\sigma_{j}^{-1}(L_{11})-\sigma_{j}^{-1}(A)}{\sigma_{j}^{-1}(A)}\leq\rho-1+\frac{\mathcal{C\rho}}{\sigma_{k}(B)}+\left(\frac{\mathcal{C\rho}}{\sigma_{k}(B)}+\rho\right)\left(\frac{\sigma_{s+1}(B)}{\sigma_{s}(B)}\right)^{2}\mathcal{O}\left(\frac{q^{\frac{5}{2}}\|R_{12}\|_{2}^{2}}{(1-\rho_{1}^{2})\sigma_{s}^2(L_{11})}\right)
\end{split}
\end{equation*}
\end{small}
with probability not less than $1-e^{-a_{2}m}-e^{-a_{2}n}-e^{-c_{2}l_{1}}-e^{-c_{2}l_{2}}-\Delta$  for all $j=1,\ldots,k$. In particular, if $s=k$, then
\begin{small}
\begin{equation*}
\begin{split}
\frac{\sigma_{j}^{-1}(L_{11})-\sigma_{j}^{-1}(A)}{\sigma_{j}^{-1}(A)}\leq\rho-1+\frac{\mathcal{C\rho}}{\sigma_{k}(B)}+\left(\frac{\mathcal{C\rho}}{\sigma_{k}(B)}+\rho\right)\left(\frac{\sigma_{k+1}(B)}{\sigma_{k}(B)}\right)^{2}\mathcal{O}\left(\frac{q^{\frac{5}{2}}\|R_{12}\|_{2}^{2}}{(1-\rho_{1}^{2})\sigma_{k}^2(L_{11})}\right)
\end{split}
\end{equation*}
\end{small}
with probability not less than $1-e^{-a_{2}m}-e^{-a_{2}n}-e^{-c_{2}l_{1}}-e^{-c_{2}l_{2}}-\Delta$  for all $j=1,\ldots,k$.
\end{cor}

\subsection{Error analysis of Algorithm 4}
\subsubsection{Matrix approximation error analysis}
We first provide a matrix approximation error bound for Algorithm 4.

\begin{thm}\label{thm:matrix-error2}
Let $A\in\mathbb{R}^{m\times n}$ be full rank and let $k$ be target rank. Suppose $k$ and $l$ are such that $l>(1+\frac{1}{\ln k})k$. Let $Q$, $L$, and $P$ be generated by Algorithm 4. Then, for any $0<\Delta\ll1$, we have
\begin{equation}
\|A-QLP^{T}\|_{2}\leq2\sqrt{\frac{a_{1}^{2}n}{c_{1}^{2}l}+1}\cdot\sigma_{k+1}(A)
\end{equation}
with probability not less than $1-e^{-a_{2}n}-e^{-c_{2}l}$;
\begin{equation}
\|A-QLP^{T}\|_{F}\leq\sqrt{\frac{k\sigma_{1}^{2}(A)\sigma_{k+1}^{2}(A)\mathcal{C}_{\Delta}^{2}}{\sigma_{k+1}^{2}(A)\mathcal{C}_{\Delta}^{2}+\sigma_{1}^{2}(A)}+\sum\limits_{i=k+1}^{n}\sigma_{i}^{2}(A)}\\
\end{equation}
with probability not less than $1-\Delta$, where
\[
\mathcal{C}_{\Delta}=\frac{e\sqrt{l}}{p+1}\left(\frac{2}{\Delta}\right)^{\frac{1}{p+1}}
\left(\sqrt{n-k}+\sqrt{l}+\sqrt{2\log\frac{2}{\Delta}}\right)
\]
and the parameters $a_{1}, c_{1}$ and $c_{2}$ are defined as in Definition \ref{def:general-gaussian} and Remark \ref{rem:parameter0}--\ref{rem:parameter}.
\end{thm}

\begin{proof}
By hypothesis, $A$ is full rank. Then $Y_{1}=A\Omega$ is full rank and thus $R=V^{T}Y_{1}$ is full rank. For  $B$, $Q$, $L$, and $P$ generated by Algorithm 4, we have
\begin{eqnarray}
\|A-QLP^{T}\|_{2}&=&\|A-VB\|_{2}=\|A-V(Y_{1}^{T}V)^{\dag}Y_{1}^{T}A\|_{2}\nonumber\\
&=&\|A-V(R^{T})^{\dag}R^{T}V^{T}A\|_{2} =\|A-VV^{T}A\|_{2},\nonumber
\end{eqnarray}
where the last equality follows from $(R^{T})^{\dag}R^{T}=I$ since $R^{T}=Y_{1}^{T}V$ has full column rank.
Since $l>\left(1+\frac{1}{\ln k}\right)k$, it follows from (\ref{equ:error1}) that
\begin{equation*}
\|A-QLP^{T}\|_{2}=\|A-VV^{T}A\|_{2}\leq 2\sqrt{\frac{a_{1}^{2}n}{c_{1}^{2}l}+1}\cdot\sigma_{k+1}(A)
\end{equation*}
with probability not less than $1-e^{-c_{2}l}-e^{-a_{2}n}$.

Using the similar proof of \eqref{equ:qlp-est} we have, for any given $0<\Delta\ll1$,
\begin{eqnarray}
\|A-QLP^{T}\|_{F}=\|A-VV^{T}A\|_{F}
%&\leq&\sqrt{\frac{k\sigma_{1}^{2}(A)\sigma_{k+1}^{2}(A)\|\hat{\Omega}_{2}\|_{2}^{2}\|\hat{\Omega}_{1}^{\dag}\|_{2}^{2}}{\sigma_{k+1}^{2}(A)\|\hat{\Omega}_{2}\|_{2}^{2}\|\hat{\Omega}_{1}^{\dag}\|_{2}^{2}+\sigma_{1}^{2}(A)}+\sum\limits_{i=k+1}^{n}\sigma_{i}^{2}(A)}\nonumber\\
\leq\sqrt{\frac{k\sigma_{1}^{2}(A)\sigma_{k+1}^{2}(A)\mathcal{C}_{\Delta}^{2}}{\sigma_{k+1}^{2}(A)\mathcal{C}_{\Delta}^{2}+\sigma_{1}^{2}(A)}+\sum\limits_{i=k+1}^{n}\sigma_{i}^{2}(A)}\nonumber
\end{eqnarray}
with probability not less than $1-\Delta$.
\end{proof}

\begin{rem}
As Remark \ref{rem:bd-pro}, for $Q$, $L$, and $P$ generated by Algorithm 4, we can easily derive that
\begin{equation}
\|A-QLP^{T}\|_{2}\leq\sqrt{\frac{k\sigma_{1}^{2}(A)\sigma_{k+1}^{2}(A)\mathcal{C}_{\Delta}^{2}}
{\sigma_{k+1}^{2}(A)\mathcal{C}_{\Delta}^{2}+\sigma_{1}^{2}(A)}+\sigma_{k+1}^{2}(A)}
\end{equation}
with probability not less than $1-\Delta$.
\end{rem}

\subsubsection{Singular value approximation analysis}
As in Section \ref{sec412}, we assume that the second QR decomposition of the QLP decomposition for the matrix $B$ generated by Algorithm 4 is without pivoting. We first give the following result on the singular value approximation error of  Algorithm 4.

\begin{thm}\label{thm:singular-error4}
Let $A$ be full rank and let $A\approx QLP^{T}$ be rank-$k$ SORQLP decomposition produced by Algorithm 4. For $B$ generated by  Algorithm 3, let $R_0$ be the $R$-factor in the pivoted QR factorization of $B$, $BP_{0}=Q_{0}R_{0}$ and let $L^T$ be the  $R$-factor in the unpivoted QR factorization of $(R_0)^T$, $R_{0}^{T}=Q_{1}L^{T}$.
Then we have
$$\sigma_{j}(L)\leq\sigma_{j}(A),\quad \forall 1\leq j\leq k.$$
Moverover,  for any $0<\Delta\ll1$, we have  $$\sigma_{j}(L)\geq\frac{\sigma_{j}(A)}{\rho}$$ with probability not less than $1-\Delta$ for all $j=1,\ldots, k$, where
\[
\rho=\sqrt{1+\mathcal{C}_{\Delta}^{2}\left(\frac{\sigma_{k+1}(A)}{\sigma_{j}(A)}\right)^{2}}\quad\mbox{with}\quad
\mathcal{C}_{\Delta}=\frac{e\sqrt{l}}{p+1}\left(\frac{2}{\Delta}\right)^{\frac{1}{p+1}}
\left(\sqrt{n-k}+\sqrt{l}+\sqrt{2\log\frac{2}{\Delta}}\right).
\]

\end{thm}

\begin{proof}
Since $A$ is full rank, $R$ is full rank. Thus, $$B=(Y_{1}^{T}V)^{\dag}Y_{1}^{T}A=(R^{T})^{\dag}R^{T}V^{T}A=V^{T}A.$$
By Lemma \ref{lem:singular-Goulb} we obtain, for $1\leq j\leq k$, $$\sigma_{j}(L)=\sigma_{j}(B)=\sigma_{j}(V^{T}A)\leq\sigma_{j}(A).$$

On the other hand, let the SVD of $A$ be given by \eqref{a:svd}. Denote
\[
V_n^T\Omega=\left[
                     \begin{array}{c}
                       \hat{\Omega}_{1} \\
                       \hat{\Omega}_{2} \\
                     \end{array}
                   \right],
\]
where $\hat{\Omega}_{1}\in\mathbb{R}^{k\times l}$, $\hat{\Omega}_{2}\in\mathbb{R}^{(n-k)\times l}$. According to Lemma \ref{lem:singular-guming}, we obtain
\[
\sigma_{j}(L)=\sigma_{j}(B)=\sigma_{j}(V^{T}A)\geq
\frac{\sigma_{j}(A)}{\sqrt{1+\|\hat{\Omega}_{2}\|_{2}^{2}\|\hat{\Omega}_{1}^{\dag}\|_{2}^{2}\left(\frac{\sigma_{k+1}(A)}
{\sigma_{j}(A)}\right)^{2}}},
\]
for all $j=1,\ldots, k$.  By Lemma \ref{lem:bound-Omega} we have, for any $0<\Delta\ll1$, $\|\hat{\Omega}_{2}\|_{2}^{2}\|\hat{\Omega}_{1}^{\dag}\|_{2}^{2}\leq\mathcal{C}_{\Delta}^{2}$ with probability not less than $1-\Delta$.
Then we have
\[
\sigma_{j}(L)=\sigma_{j}(V^{T}A)\geq\frac{\sigma_{j}(A)}{\rho}
\]
with probability not less than $1-\Delta$, where $\rho=\sqrt{1+\mathcal{C}_{\Delta}^{2}\left(\frac{\sigma_{k+1}(A)}{\sigma_{j}(A)}\right)^{2}}$.
\end{proof}

Next, we give the following result on the largest singular value approximation
error of Algorithm 4.

\begin{thm}\label{thm:singular-error5}
Let $A\approx QLP^{T}$ be rank-$k$ SORQLP decomposition produced by Algorithm 4. For $B$ generated by  Algorithm 4, let $R_0$ be the $R$-factor in the pivoted QR factorization of $B$, $BP_{0}=Q_{0}R_{0}$ and let $L^T$ be the  $R$-factor in the unpivoted QR factorization of $(R_0)^T$, $R_{0}^{T}=Q_{1}L^{T}$.
Partition $R_{0}$ and $L$ as
\[
R_{0}=\left[
                                                                                    \begin{array}{cc}
                                                                                      r_{11} & R_{12} \\
                                                                                      0 & R_{22} \\
                                                                                    \end{array}
                                                                                  \right]\quad\mbox{and}\quad L=\left[
                                                                                             \begin{array}{cc}
                                                                                               l_{11} & 0 \\
                                                                                               L_{21} & L_{22} \\
                                                                                             \end{array}
                                                                                           \right],
\]
where $r_{11}$, $l_{11}\in\mathbb{R}^{1\times1}$. Assume $\sigma_{1}(B)>\sigma_{2}(B)$, $\|R_{22}\|_{2}\leq\sqrt{2(l-1)}\sigma_{2}(B)$ and $\frac{\|R_{22}\|_{2}}{|r_{11}|}<1$.  Then, for any given $0<\Delta\ll1$, we have
\begin{equation}
|l_{11}|^{-1}-\sigma_{1}^{-1}(A)\leq\frac{1}{\sigma_{1}(B)}\left(\frac{\rho-1}{\rho}\right)+\frac{\sigma_{2}^{2}(B)}{\sigma_{1}^{3}(B)}\mathcal{O}\left(\frac{q^{\frac{5}{2}}\|R_{12}\|_{2}^{2}}{\left(1-\rho_{1}^{2}\right)l_{11}^{2}}\right)
\end{equation}
with probability not less than $1-\Delta$, where $\rho_{1}=\frac{\|L_{22}\|_{2}}{|l_{11}|}$,
\[
\rho=\sqrt{1+\mathcal{C}_{\Delta}^{2}\left(\frac{\sigma_{k+1}(A)}{\sigma_{j}(A)}\right)^{2}}
\]
with
\[
\mathcal{C}_{\Delta}=\frac{e\sqrt{l}}{p+1}\left(\frac{2}{\Delta}\right)^{\frac{1}{p+1}}\left(\sqrt{n-k}
+\sqrt{l}+\sqrt{2\log\frac{2}{\Delta}}\right).
\]
\end{thm}

\begin{proof}
From Theorem \ref{thm:singular-error4} we have $$\sigma_{1}(B)=\sigma_{1}(L)\geq\frac{\sigma_{1}(A)}{\rho},\quad\mbox{i.e.,}\quad\frac{1}{\sigma_{1}(A)}\geq\frac{1}{\rho\sigma_{1}(B)}$$
with probability not less than $1-\Delta$.
By Lemma \ref{lem:l11-smallsingular} we obtain $$|l_{11}|^{-1}\leq\sigma_{1}^{-1}(B)+\frac{\sigma_{2}^{2}(B)}{\sigma_{1}^{3}(B)}\mathcal{O}\left(\frac{q^{\frac{5}{2}}\|R_{12}\|_{2}^{2}}{\left(1-\rho_{1}^{2}\right)l_{11}^{2}}\right).$$
Subtracting $\sigma_{1}^{-1}(A)$ from the both sides of the above inequality yields
\begin{eqnarray}
|l_{11}|^{-1}-\sigma_{1}^{-1}(A)&\leq&\sigma_{1}^{-1}(B)-\frac{1}{\sigma_{1}(A)}+\frac{\sigma_{2}^{2}(B)}{\sigma_{1}^{3}(B)}\mathcal{O}\left(\frac{q^{\frac{5}{2}}\|R_{12}\|_{2}^{2}}{\left(1-\rho_{1}^{2}\right)l_{11}^{2}}\right)\nonumber\\
&\leq&\sigma_{1}^{-1}(B)-\frac{1}{\rho\sigma_{1}(B)}+\frac{\sigma_{2}^{2}(B)}{\sigma_{1}^{3}(B)}\mathcal{O}\left(\frac{q^{\frac{5}{2}}\|R_{12}\|_{2}^{2}}{\left(1-\rho_{1}^{2}\right)l_{11}^{2}}\right)\nonumber\\
&=&\sigma_{1}^{-1}(B)\left(1-\frac{1}{\rho}\right)+\frac{\sigma_{2}^{2}(B)}{\sigma_{1}^{3}(B)}\mathcal{O}\left(\frac{q^{\frac{5}{2}}\|R_{12}\|_{2}^{2}}{\left(1-\rho_{1}^{2}\right)l_{11}^{2}}\right)\nonumber
\end{eqnarray}
with probability not less than $1-\Delta$.
\end{proof}

Next, we present the following result on  the interior singular value approximation errors of  Algorithm 4.
\begin{thm}\label{thm:singular-error6}
Let $A\approx QLP^{T}$ be rank-$k$ SORQLP decomposition produced by Algorithm 4. For $B$ generated by  Algorithm 4, let $R_0$ be the $R$-factor in the pivoted QR factorization of $B$, $BP_{0}=Q_{0}R_{0}$ and let $L^T$ be the  $R$-factor in the unpivoted QR factorization of $(R_0)^T$, $R_{0}^{T}=Q_{1}L^{T}$. Suppose $\sigma_{s}(B)>\sigma_{s+1}(B)$ for some $1\leq s<k$. Partition $R_{0}$ and $L$ as
\[
R_{0}=\left[
                                   \begin{array}{cc}
                                     R_{11} & R_{12} \\
                                     0 & R_{22} \\
                                   \end{array}
                                 \right], L=\left[
                                              \begin{array}{cc}
                                                L_{11} & 0 \\
                                                L_{21} & L_{22} \\
                                              \end{array}
                                            \right],
\]
where $R_{11}$, $L_{11}\in\mathbb{R}^{s\times s}$. Assume that $\|R_{22}\|_{2}\leq\sqrt{(s+1)(l-s)}\sigma_{s+1}(B)$, $\sigma_{s}(R_{11})\geq\frac{\sigma_{s}(B)}{\sqrt{s(l-s+1)}}$, and $\frac{\|R_{22}\|_{2}}{\sigma_s(R_{11})}<1$.  Then, for any given $0<\Delta\ll1$, we have
\begin{equation}\label{thm48:42}
\frac{\sigma_{j}^{-1}(L_{11})-\sigma_{j}^{-1}(A)}{\sigma_{j}^{-1}(A)}\leq\rho-1+\rho\frac{\sigma_{s+1}^{2}(B)}{\sigma_{s}^{2}(B)}\mathcal{O}\left(\frac{q^{\frac{5}{2}}\|R_{12}\|_{2}^{2}}{(1-\rho_{1}^{2})\sigma_{s}^2(L_{11})}\right)
\end{equation}
with probability not less than $1-\Delta$  for all $i=1,\ldots,s$ and
\begin{equation}\label{thm48:43}
\frac{\sigma_{j}(L_{22})-\sigma_{s+j}(A)}{\sigma_{s+j}(A)}\leq\frac{\sigma_{s+j}(B)}{\sigma_{s+j}(A)}-1+\frac{\sigma_{s+j}(B)}{\sigma_{s+j}(A)}\left(\frac{\sigma_{s+1}(B)}{\sigma_{s}(B)}\right)^{2}\mathcal{O}\left(\frac{q^{\frac{5}{2}}\|R_{12}\|_{2}^{2}}{(1-\rho_{1}^{2})\sigma_{s}^2(L_{11})}\right),
\end{equation}
for all $j=1,\ldots,k-s$, where $\rho_{1}=\frac{\|L_{22}\|_{2}}{\sigma_{s}(L_{11})}$ and
\[
\rho=\sqrt{1+\mathcal{C}_{\Delta}^{2}\left(\frac{\sigma_{k+1}(A)}{\sigma_{j}(A)}\right)^{2}}
\]
with
\[
\mathcal{C}_{\Delta}=\frac{e\sqrt{l}}{p+1}\left(\frac{2}{\Delta}\right)^{\frac{1}{p+1}}\left(\sqrt{n-k}
+\sqrt{l}+\sqrt{2\log\frac{2}{\Delta}}\right).
\]
\end{thm}

\begin{proof}
From Theorem \ref{thm:singular-error4} we have $\sigma_{j}(B)\leq\sigma_{j}(A)$, i.e., $\frac{1}{\sigma_{j}(A)}\leq\frac{1}{\sigma_{j}(B)}$ for all $j=1,\ldots, k$. Similarly, $\sigma_{j}(A)\leq\rho\sigma_{j}(B)$ with probability not less than $1-\Delta$  for all $j=1,\ldots, k$.
From Lemma \ref{lem:block-error} we have, for $j=1,\ldots,k-s$,
\begin{equation*}
\begin{split}
\sigma_{j}(L_{22})-\sigma_{s+j}(A)&\leq\sigma_{s+j}(B)-\sigma_{s+j}(A)+\frac{\sigma_{s+1}^{2}(B)\sigma_{s+j}(B)}{\sigma_{s}^{2}(B)}\mathcal{O}\left(\frac{q^{\frac{5}{2}}\|R_{12}\|_{2}^{2}}{(1-\rho_{1}^{2})\sigma_{s}^2(L_{11})}\right).\\
\end{split}
\end{equation*}
Dividing $\sigma_{s+j}(A)$ on the both side of the above inequality yields
\begin{eqnarray}
\frac{\sigma_{j}(L_{22})-\sigma_{s+j}(A)}{\sigma_{s+j}(A)}&\leq&\frac{\sigma_{s+j}(B)-\sigma_{s+j}(A)}{\sigma_{s+j}(A)}+\frac{\sigma_{s+1}^{2}(B)\sigma_{s+j}(B)}{\sigma_{s+j}(A)\sigma_{s}^{2}(B)}\mathcal{O}\left(\frac{q^{\frac{5}{2}}\|R_{12}\|_{2}^{2}}{(1-\rho_{1}^{2})\sigma_{s}^2(L_{11})}\right)\nonumber\\
&\leq&\frac{\sigma_{s+j}(B)}{\sigma_{s+j}(A)}-1+\frac{\sigma_{s+j}(B)}{\sigma_{s+j}(A)}\left(\frac{\sigma_{s+1}(B)}{\sigma_{s}(B)}\right)^{2}\mathcal{O}\left(\frac{q^{\frac{5}{2}}\|R_{12}\|_{2}^{2}}{(1-\rho_{1}^{2})\sigma_{s}^2(L_{11})}\right),\nonumber
%&\leq\frac{\sigma_{s+1}^{2}(B)}{\sigma_{s}^{2}(B)}\mathcal{O}\left(\frac{q^{\frac{5}{2}}\|R_{12}\|_{2}^{2}}{(1-\rho_{1}^{2})\sigma_{s}^2(L_{11})}\right).\\
\end{eqnarray}
for all $j=1,\ldots,k-s$. This shows that \eqref{thm48:43} holds.

On the other hand, by Lemma \ref{lem:bound-Omega} we have, for $j=1,\ldots,s$,
$$\sigma_{j}^{-1}(L_{11})-\sigma_{j}^{-1}(A)\leq\sigma_{j}^{-1}(B)-\sigma_{j}^{-1}(A)+\frac{\sigma_{s+1}^{2}(B)}{\sigma_{s}^{2}(B)\sigma_{j}(B)}\mathcal{O}\left(\frac{q^{\frac{5}{2}}\|R_{12}\|_{2}^{2}}{(1-\rho_{1}^{2})\sigma_{s}^2(L_{11})}\right).$$
Dividing $\sigma_{j}^{-1}(A)$ on the both sides of the above inequality gives rise to
\begin{eqnarray}
\frac{\sigma_{j}^{-1}(L_{11})-\sigma_{j}^{-1}(A)}{\sigma_{j}^{-1}(A)}&\leq&\frac{\sigma_{j}^{-1}(B)-\sigma_{j}^{-1}(A)}{\sigma_{j}^{-1}(A)}+\frac{\sigma_{j}(A)\sigma_{s+1}^{2}(B)}{\sigma_{s}^{2}(B)\sigma_{j}(B)}\mathcal{O}\left(\frac{q^{\frac{5}{2}}\|R_{12}\|_{2}^{2}}{(1-\rho_{1}^{2})\sigma_{s}^2(L_{11})}\right)\nonumber\\
&\leq&\frac{\sigma_{j}^{-1}(B)-(\rho\sigma_{j}(B))^{-1}}{(\rho\sigma_{j}(B))^{-1}}+\rho\frac{\sigma_{j}(B)\sigma_{s+1}^{2}(B)}{\sigma_{s}^{2}(B)\sigma_{j}(B)}\mathcal{O}\left(\frac{q^{\frac{5}{2}}\|R_{12}\|_{2}^{2}}{(1-\rho_{1}^{2})\sigma_{s}^2(L_{11})}\right)\nonumber\\
&\leq&\rho-1+\rho\frac{\sigma_{s+1}^{2}(B)}{\sigma_{s}^{2}(B)}\mathcal{O}\left(\frac{q^{\frac{5}{2}}\|R_{12}\|_{2}^{2}}{(1-\rho_{1}^{2})\sigma_{s}^2(L_{11})}\right)\nonumber
\end{eqnarray}
with probability not less than $1-\Delta$ for all $j=1,\ldots,s$. This shows that \eqref{thm48:42} holds.
\end{proof}

\begin{cor}
Under the same assumptions of Theorem \ref{thm:singular-error6}, if $s>k$,  then
\begin{equation*}
\begin{split}
\frac{\sigma_{j}^{-1}(L_{11})-\sigma_{j}^{-1}(A)}{\sigma_{j}^{-1}(A)}\leq\rho-1+\rho\frac{\sigma_{s+1}^{2}(B)}{\sigma_{s}^{2}(B)}\mathcal{O}\left(\frac{q^{\frac{5}{2}}\|R_{12}\|_{2}^{2}}{(1-\rho_{1}^{2})\sigma_{s}^2(L_{11})}\right)
\end{split}
\end{equation*}
with probability not less than $1-\Delta$ for all $j=1,\ldots,k$. In particular, if $s=k$, then
\begin{equation*}
\begin{split}
\frac{\sigma_{j}^{-1}(L_{11})-\sigma_{j}^{-1}(A)}{\sigma_{j}^{-1}(A)}\leq\rho-1+\rho\frac{\sigma_{k+1}^{2}(B)}{\sigma_{k}^{2}(B)}\mathcal{O}\left(\frac{q^{\frac{5}{2}}\|R_{12}\|_{2}^{2}}{(1-\rho_{1}^{2})\sigma_{k}^2(L_{11})}\right)
\end{split}
\end{equation*}
with probability not less than $1-\Delta$  for all $j=1,\ldots,k$.
\end{cor}

\section{Numerical experiments}
In this section, we give some numerical examples to illustrate the effectiveness of Algorithms 3--4. We also compare our algorithms with Algorithm 1 from \cite{Stewart1999QLP} and Algorithm 2 from \cite{Xiang2020RQLP}. All experiment are performed  by using {\tt MATLAB 2019b} on a personal laptop with an Intel(R) CPU i5-10210U of 1.6 GHz and 8 GB of RAM.
\begin{exa}\label{ex:51} \cite{Tropp2017sketching}(synthetic input matrix)
Let $A=U_n\Sigma V_n^{T}\in\mathbb{R}^{n\times n}$, where $U_n,V_n\in\mathbb{R}^{n\times n}$ are orthogonal matrices generated by using the built-in functions {\tt randn} and {\tt orth} and $\Sigma$ as follows:
\begin{itemize}
\item Polynomially decaying spectrum (\textsf{pds}): $$\Sigma=\diag(1,\ldots,1,2^{-s},3^{-s},\ldots,(n-t+1)^{-s});$$
\item Exponentially decaying spectrum (\textsf{eds}): $$\Sigma=\diag(1,\ldots,1,2^{-s},2^{-2s},\ldots,2^{-(n-t)s}).$$
\end{itemize}
Here, the constants $t,s>0$ control the rank of the significant part of the matrix and the rate of decay, respectively. We report our numerical results for  $n=2000$.
\end{exa}

\begin{exa}\label{ex:52}\cite{Hansen2007tools}(ill-conditioned matrix)
The ill-conditioned matrix $A$ is generated by discretization of the Fredholm integral equation of the first kind with square integrable kernel: $$\int_{z_{1}}^{z_{2}}K(y,z)f(z)dz=g(y),~y_{1}\leq y\leq y_{2},$$ where $y_{1},y_{2},z_{1}$, and $z_{2}$ are some constants. The Galerkin discretization method is employed and the examples \textsf{heat} and \textsf{deriv2} are utilized. In our numerical experiments, we set the test matrix size to $2000\times 2000$.
\end{exa}

In Algorithm 3 and Algorithm 4, the rank of the output matrix is $l=k+p$, hence the low-rank representation of $A$ is $$\hat{A}=Q(:,1:k)L(1:k,1:k)(P(:,1:k))^{T}.$$
Then the relative matrix approximation error is given by $$E_{F}=\frac{\|A-\hat{A}\|_{F}}{\|A\|_{F}}$$ and the absolute and relative  singular value approximation errors  can be measured by $$ AE_{\sigma_j}=|\sigma_{j}(A)-|l_{jj}|| \quad\mbox{and}\quad
RE_{\sigma_j}=\frac{|\sigma_{j}(A)-|l_{jj}||}{\sigma_{j}(A)},\quad \forall 1\le j\le k.$$
Specifically, for the singular value approximation error, we take $k=860$. The parameters used in Algorithms 2--4 are listed in Table \ref{tab:1}.

\begin{table}
\caption{Parameters used in Algorithms 2--4}
\label{tab:1}
\begin{center}
\begin{tabular}{c|c}
  \hline
  Method & Parameters \\
  \hline
  Algorithm 2 & oversampling parameter $p=5$, $l=k+p$ \\
  Algorithm 3 & oversampling parameter $p=5$, $l_{1}=k+p$, $l_{2}=2k$ \\
  Algorithm 4 & oversampling parameter $p=5$, $l=k+p$ \\
  \hline
\end{tabular}
\end{center}
\end{table}

\subsection{Comparison of runing time}
Figure \ref{fig:CPU} shows the running time of the four algorithms with different target ranks. From Figure \ref{fig:CPU}, we can find that the three randomized method, i.e., Algorithms 2--4 are always much faster than Algorithm 1. In the comparison of randomized algorithms, Algorithm 2 are slightly faster than Algorithm 3 and Algorithm4. However, we note that the time of Figure \ref{fig:CPU} is not included data communication time. Thus, the total computational cost of two single-pass algorithms, i.e., Algorithms 3--4 are cheaper than Algorithm 2. Especially when matrix size is relatively large, the overall speed difference of algorithms is more obvious.

\begin{figure}[H]%[htbp]
\centering
\subfigure[Example \ref{ex:51}: \textsf{pds}($t=30, s=2$)]{
\includegraphics[width=6.5cm]{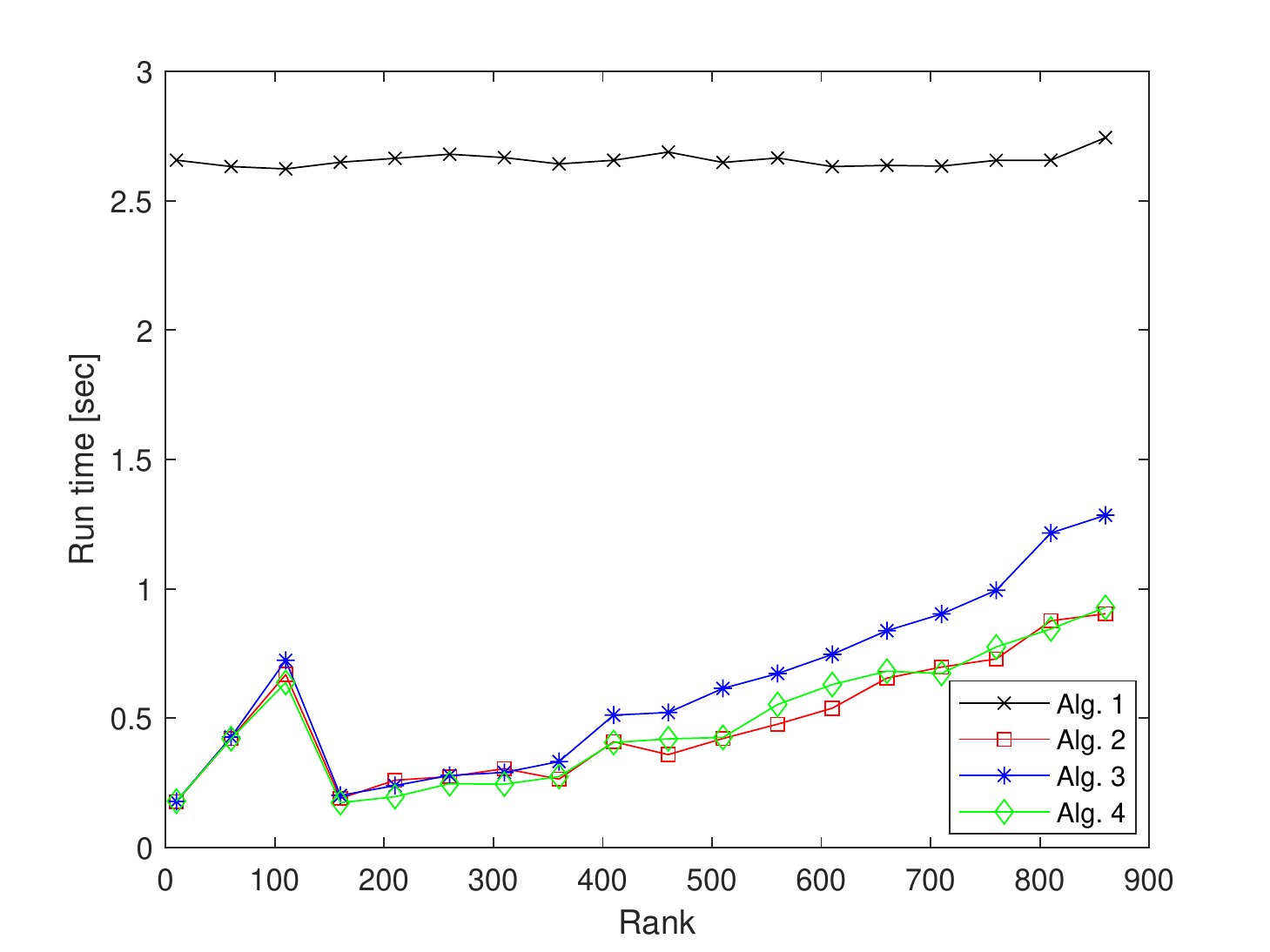}
%\caption{fig1}
}
\quad
\subfigure[Example \ref{ex:51}: \textsf{eds}($t=30, s=0.25$)]{
\includegraphics[width=6.5cm]{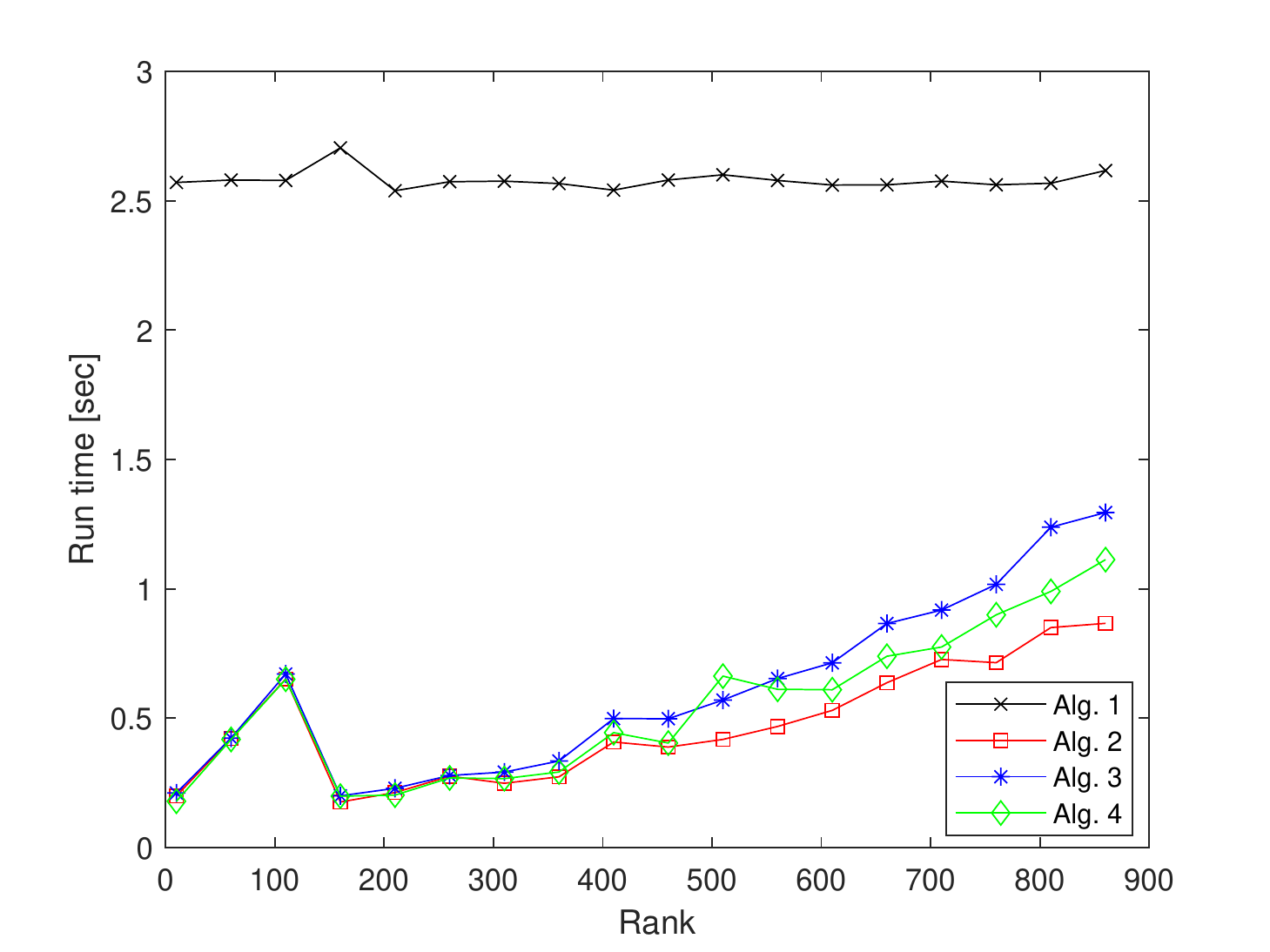}
}
\quad
\subfigure[Example \ref{ex:52}: \textsf{heat}]{
\includegraphics[width=6.5cm]{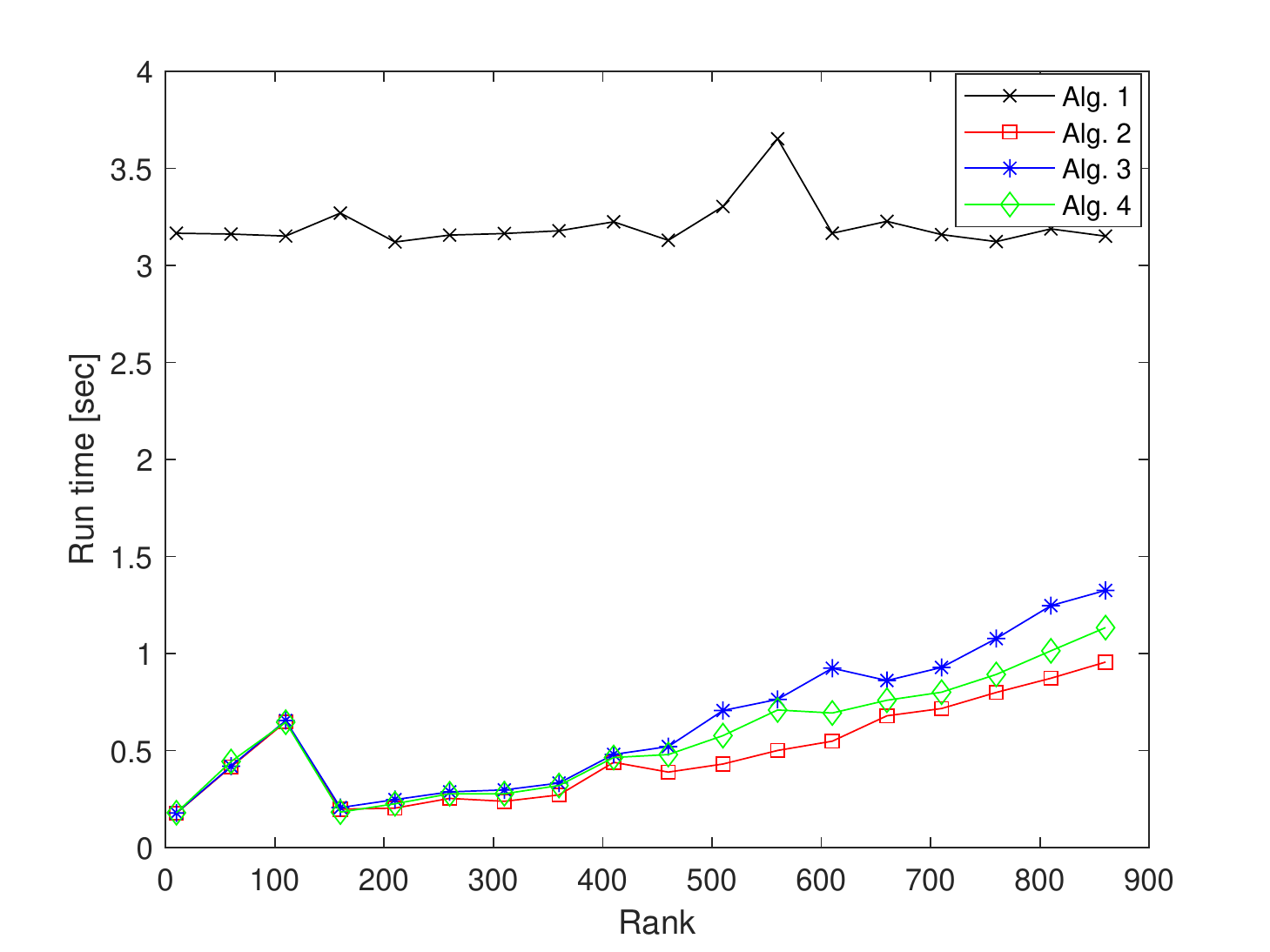}
}
\quad
\subfigure[Example \ref{ex:52}: \textsf{deriv2}]{
\includegraphics[width=6.5cm]{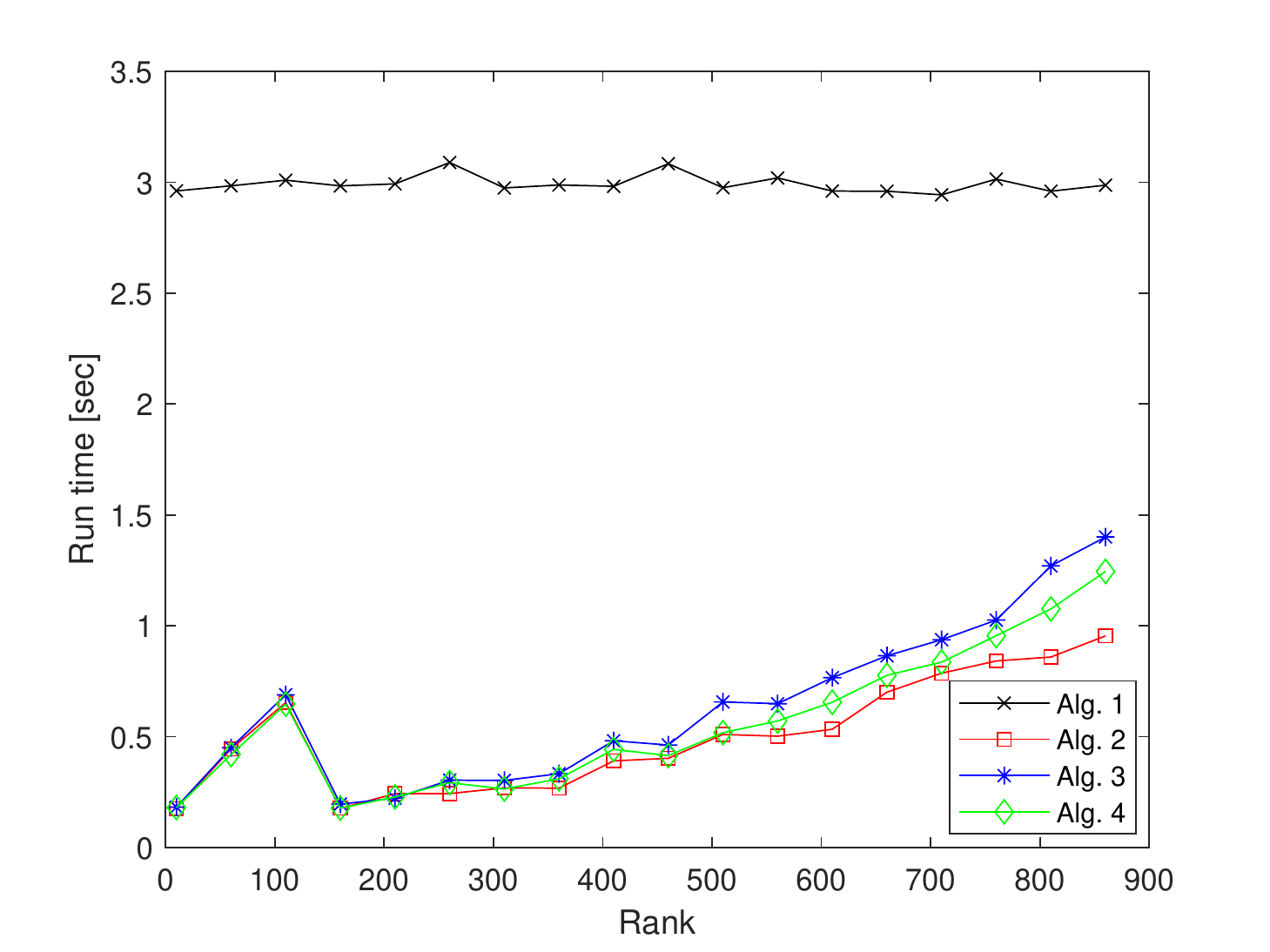}
}
\caption{Running time for a fixed $A$.}
\label{fig:CPU}
\end{figure}

\subsection{Comparison of matrix approximation error}
Figure \ref{fig:Merror} shows the trend of matrix approximation error with different target rank. We can find a very interesting phenomenon in Figure \ref{fig:Merror}. For Algorithms 1--3, in terms of matrix approximation error, they all show good performance in four numerical examples, and the error decreases with the increase of matrix rank, which is consistent with our analysis in this paper. Since Algorithm 4 algorithm is only suitable for full rank cases, we only consider example \textsf{pds} and \textsf{deriv2}. In these two numerical examples, the matrix approximation error of Algorithm 4 is not significantly different from that of the other three algorithms, and it is even slightly better than Algorithm 4.

\begin{figure}[H]%[htbp]
\centering
\subfigure[Example \ref{ex:51}: \textsf{pds}($t=30, s=2$)]{
\includegraphics[width=6.5cm]{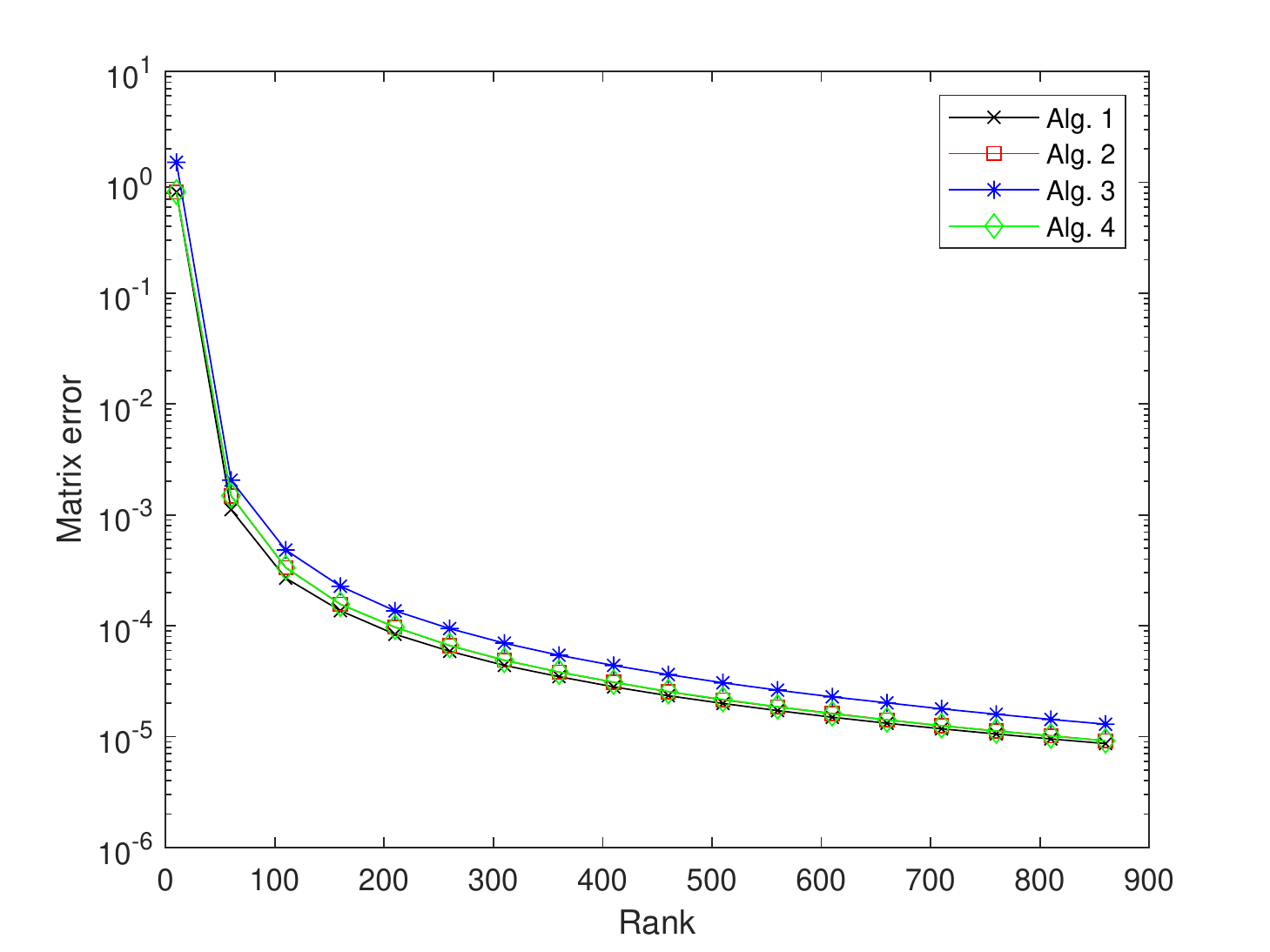}
%\caption{fig1}
}
\quad
\subfigure[Example \ref{ex:51}: \textsf{eds}($t=30, s=0.25$)]{
\includegraphics[width=6.5cm]{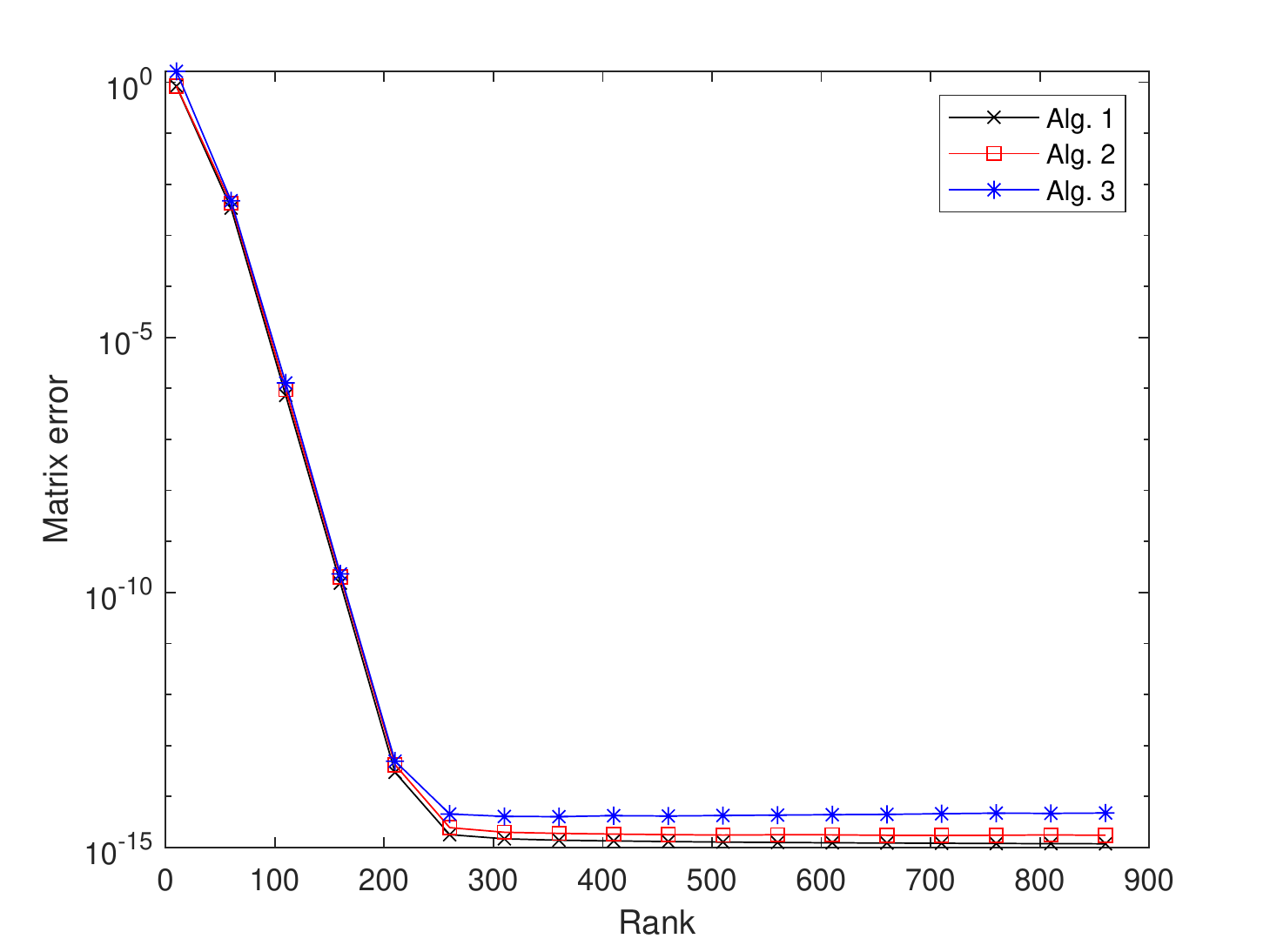}
}
\quad
\subfigure[Example \ref{ex:52}: \textsf{heat}]{
\includegraphics[width=6.5cm]{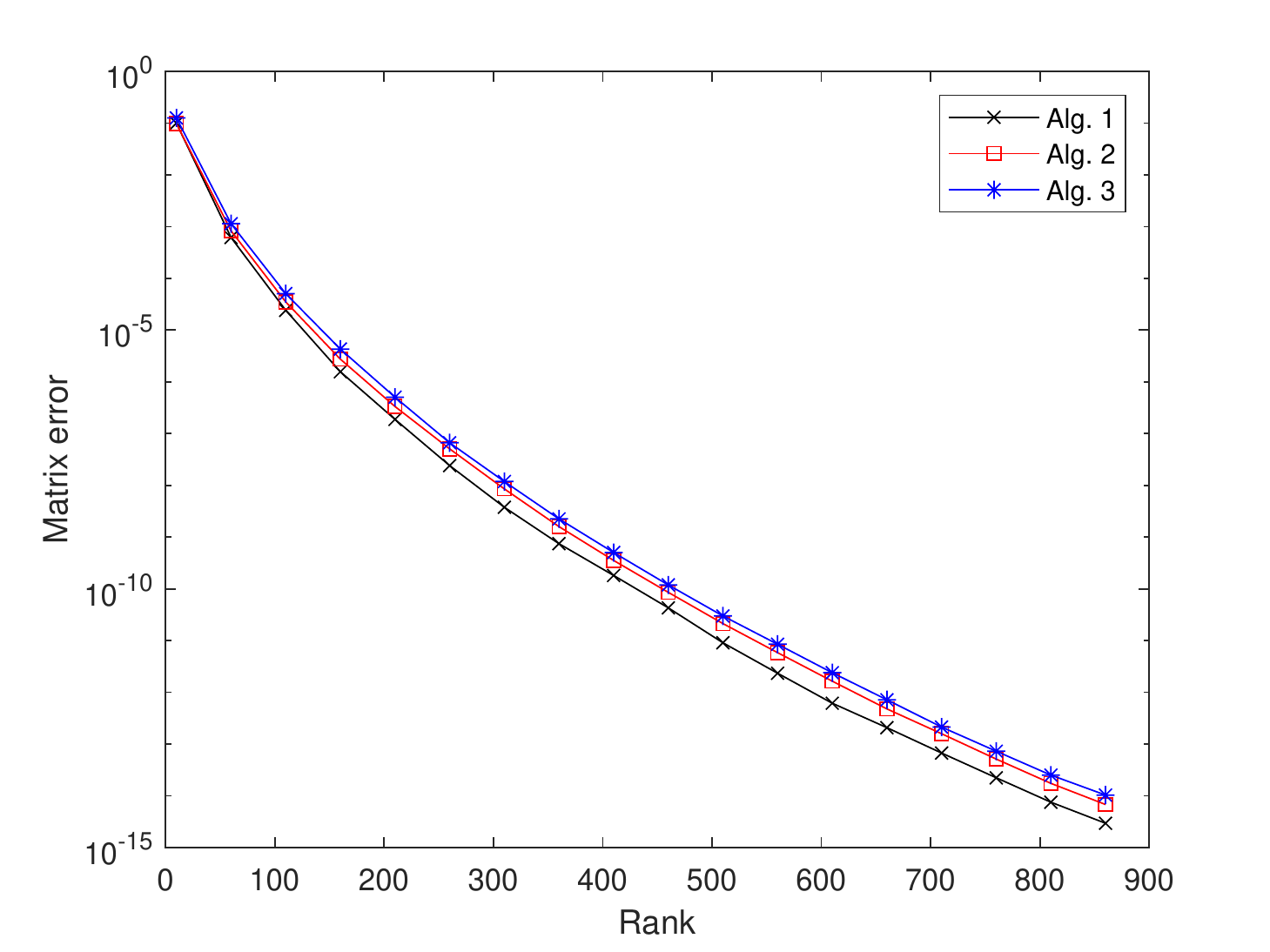}
}
\quad
\subfigure[Example \ref{ex:52}: \textsf{deriv2}]{
\includegraphics[width=6.5cm]{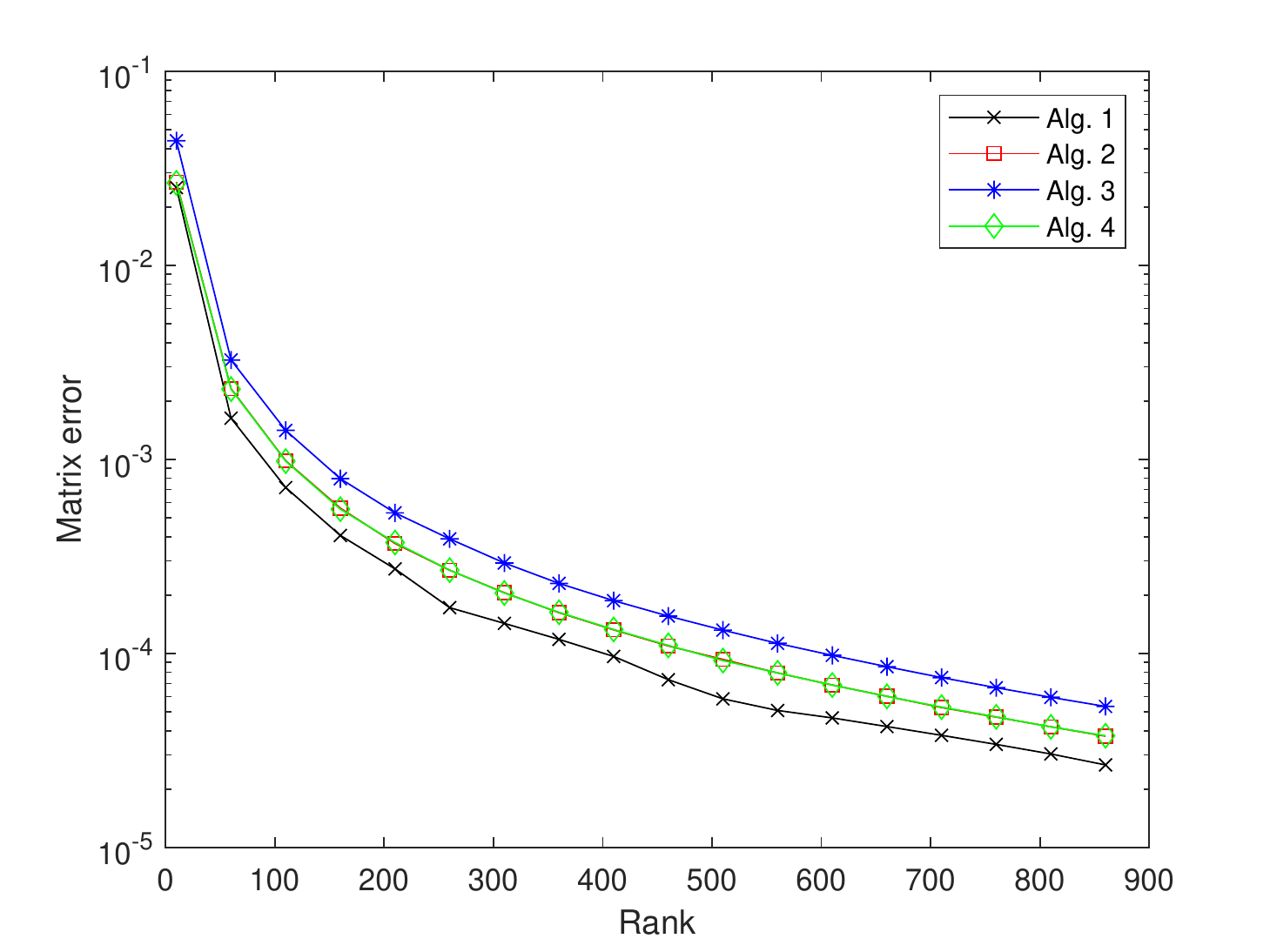}
}
\caption{Matrix approximation error for a fixed $A$.}
\label{fig:Merror}
\end{figure}

\subsection{Comparison of singular value approximation error}
In Figures \ref{fig:Serror_a}--\ref{fig:Serror_top}, we plot the curves of singular value absolute errors, singular value relative errors and top 30 singular value relative errors for different algorithms, respectively. In terms of absolute error and relative error of singular value, as shown in Figure \ref{fig:Serror_a} and \ref{fig:Serror_ar}, Algorithm 3 has a very good approximation effect on matrix singular value. Most of the approximation effect is close to Algorithm 1, and even some singular value approximation effect is better. For Algorithm 4, we know that the matrix approximation effect will be worse in the case of not full rank, but the approximation effect of this algorithms for large singular values of matrix is similar to other algorithms. In the example \textsf{eds}, the relative error of the singular value suddenly increases because the singular value itself is smaller than the machine accuracy. In Figure \ref{fig:Serror_top}, we find that the relative errors of top 30 singular values of the two single-pass randomized algorithms are very close to those of Algorithm 1, and even the singular value relative errors is exactly the same as Algorithm 1 and Algorithm 2.
\begin{figure}[H]%[htbp]
\centering
\subfigure[Example  \ref{ex:51}: \textsf{pds}($t=30, s=2$)]{
\includegraphics[width=6.5cm]{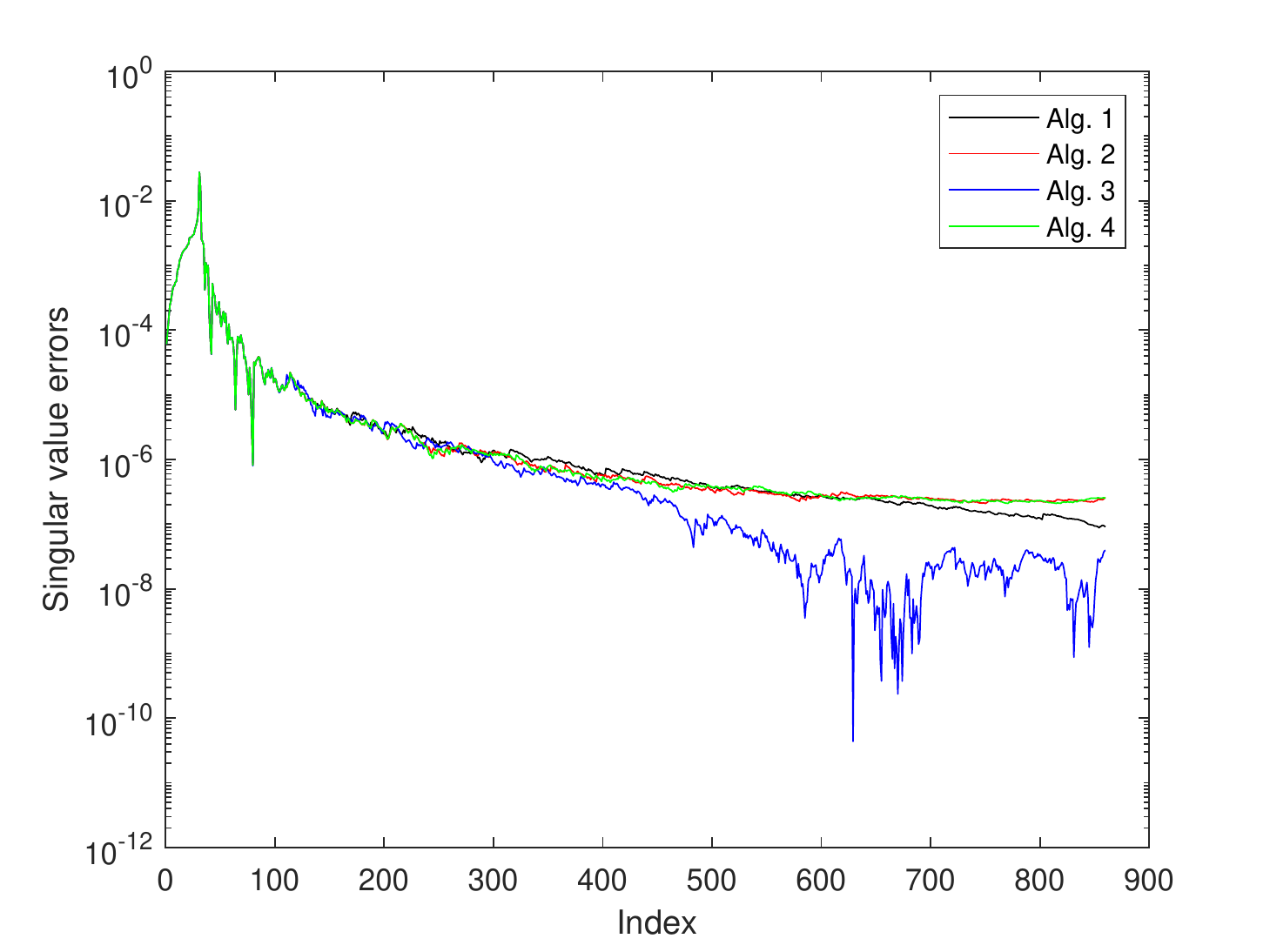}
%\caption{fig1}
}
\quad
\subfigure[Example  \ref{ex:51}: \textsf{eds}($t=30, s=0.25$)]{
\includegraphics[width=6.5cm]{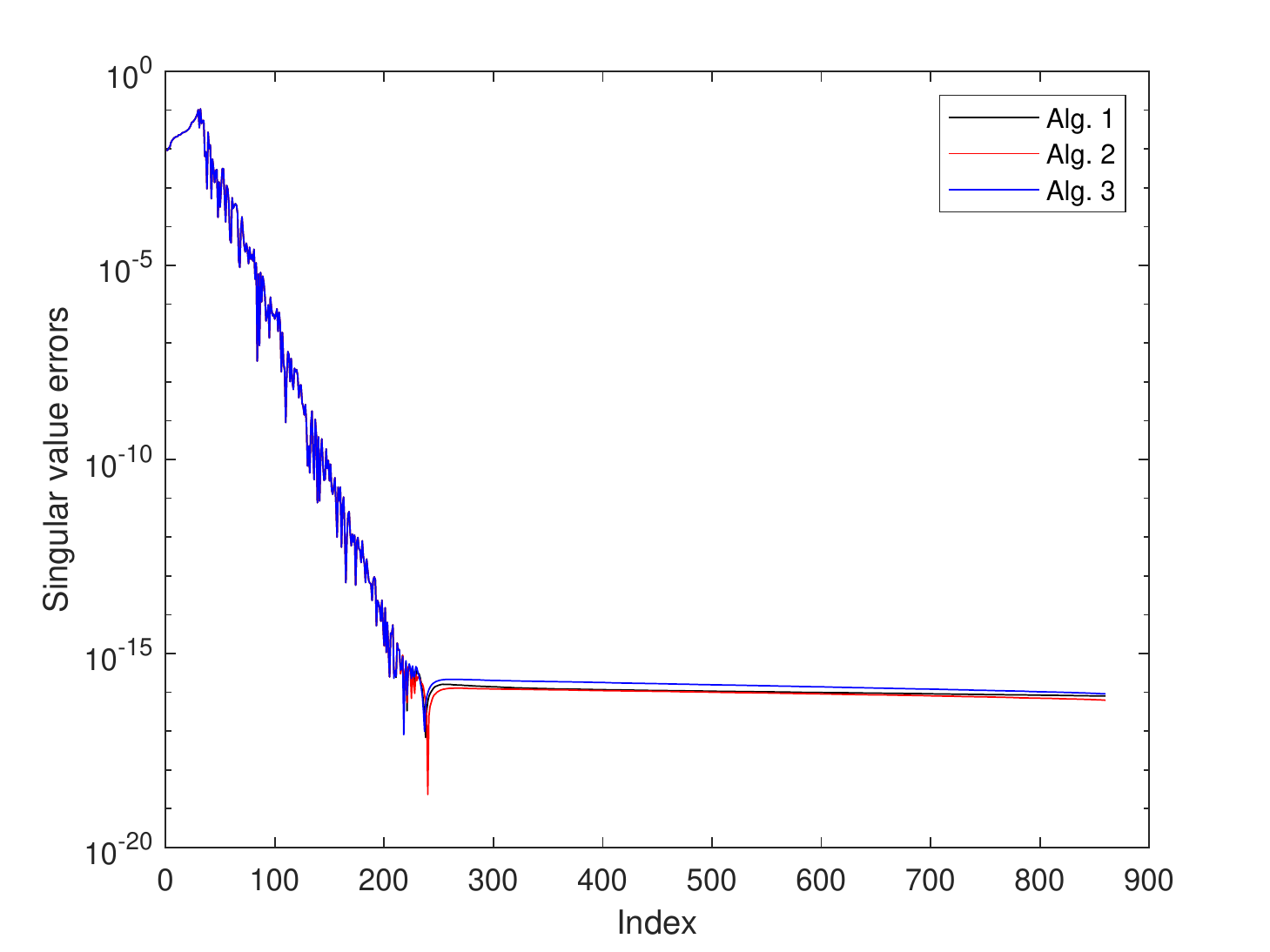}
}
\quad
\subfigure[Example  \ref{ex:52}: \textsf{heat}]{
\includegraphics[width=6.5cm]{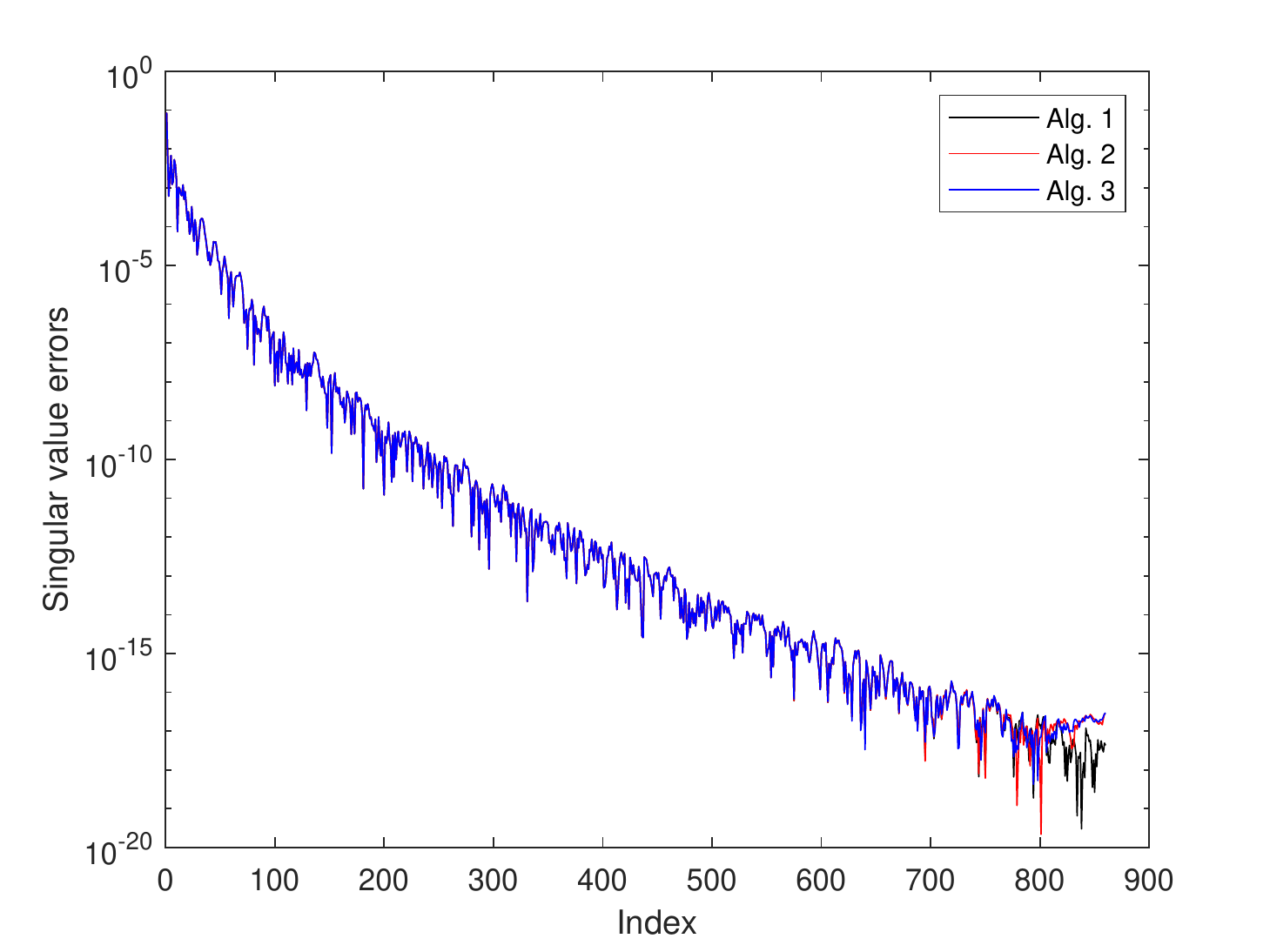}
}
\quad
\subfigure[Example  \ref{ex:52}: \textsf{deriv2}]{
\includegraphics[width=6.5cm]{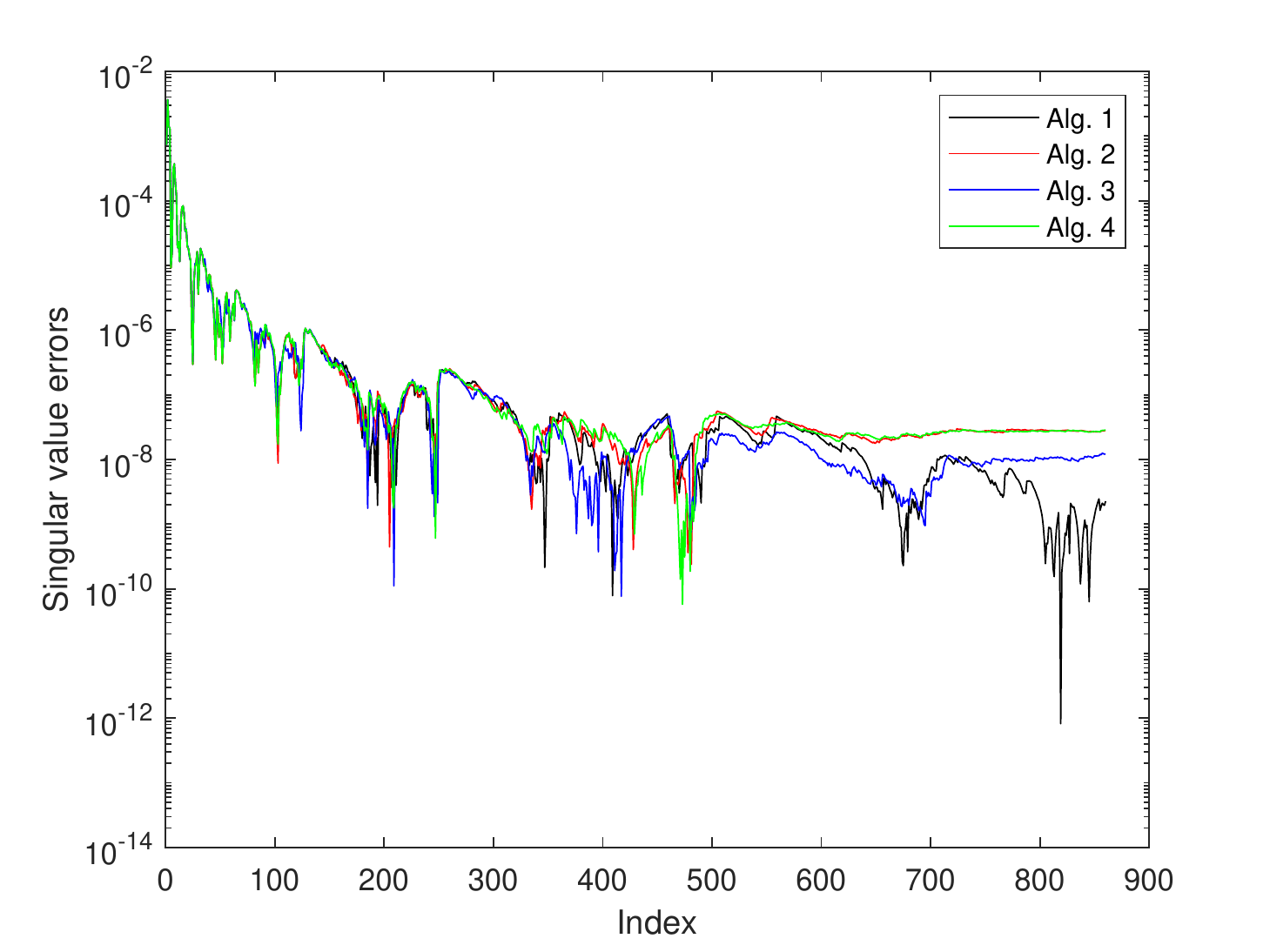}
}
\caption{Singular value absolute errors for a fixed $A$.}
\label{fig:Serror_a}
\end{figure}

\begin{figure}[H]%[htbp]
\centering
\subfigure[Example \ref{ex:51}: \textsf{pds}($t=30, s=2$)]{
\includegraphics[width=6.5cm]{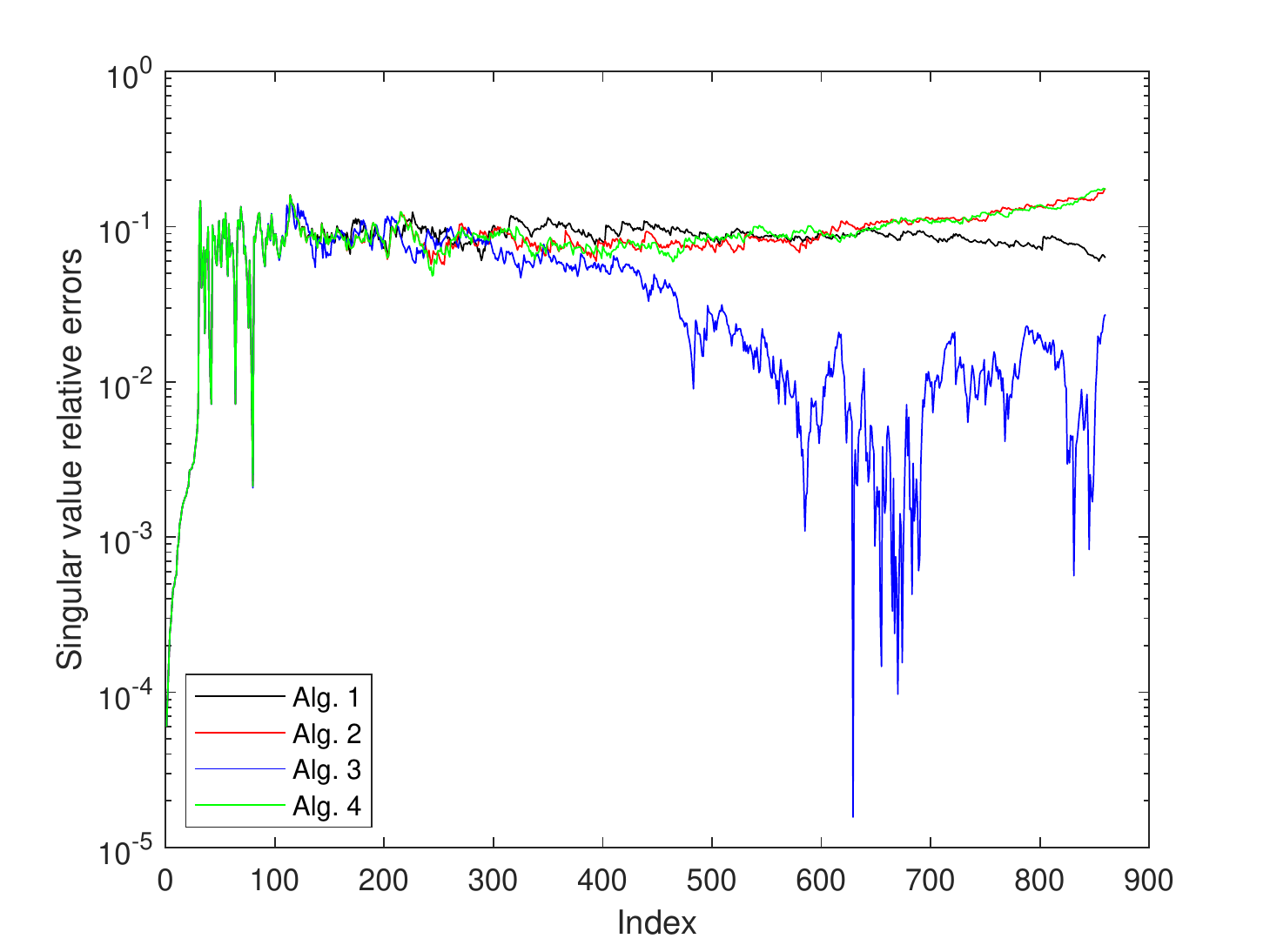}
%\caption{fig1}
}
\quad
\subfigure[Example \ref{ex:51}: \textsf{eds}($t=30, s=0.25$)]{
\includegraphics[width=6.5cm]{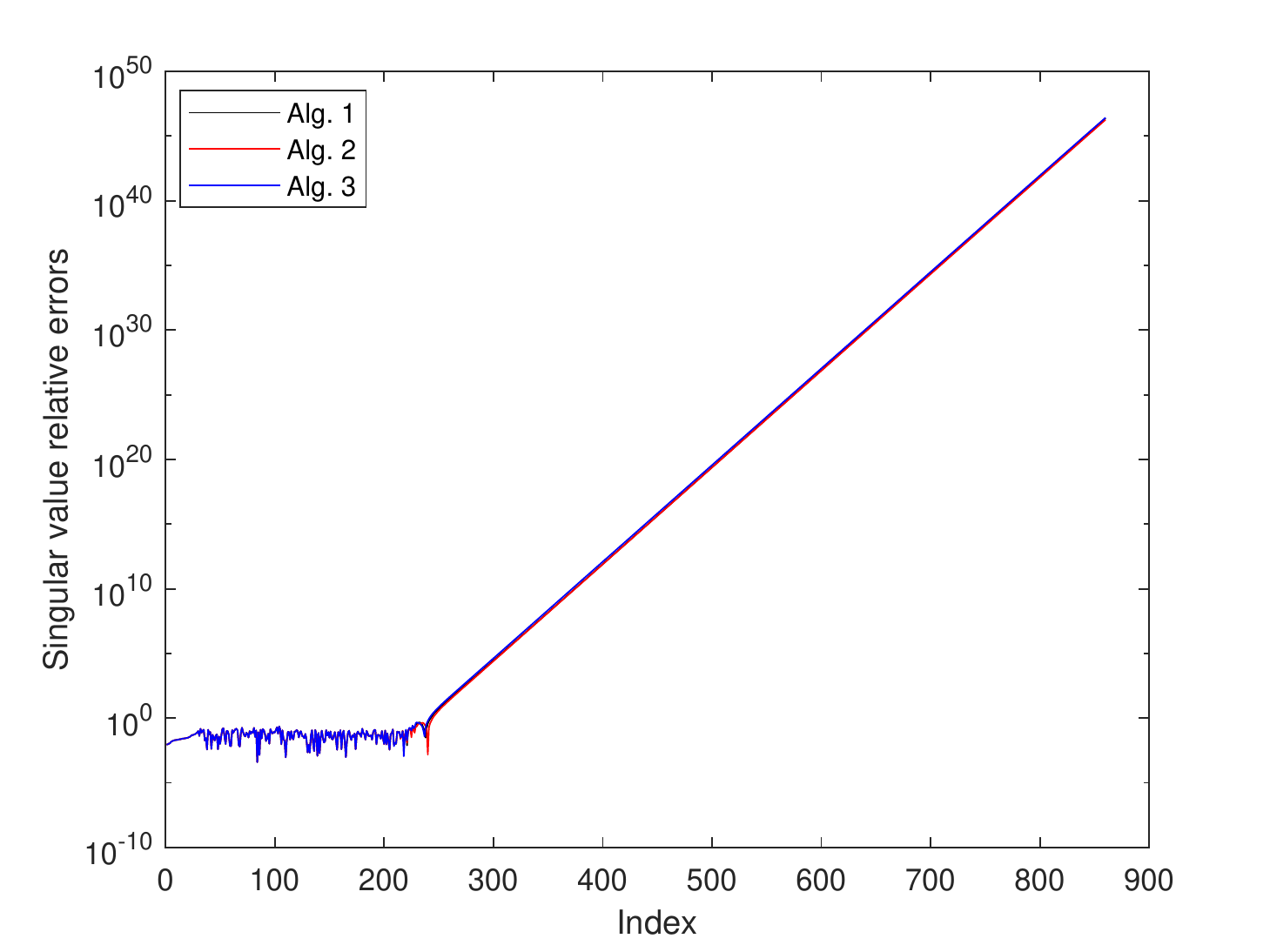}
}
\quad
\subfigure[Example \ref{ex:52}: \textsf{heat}]{
\includegraphics[width=6.5cm]{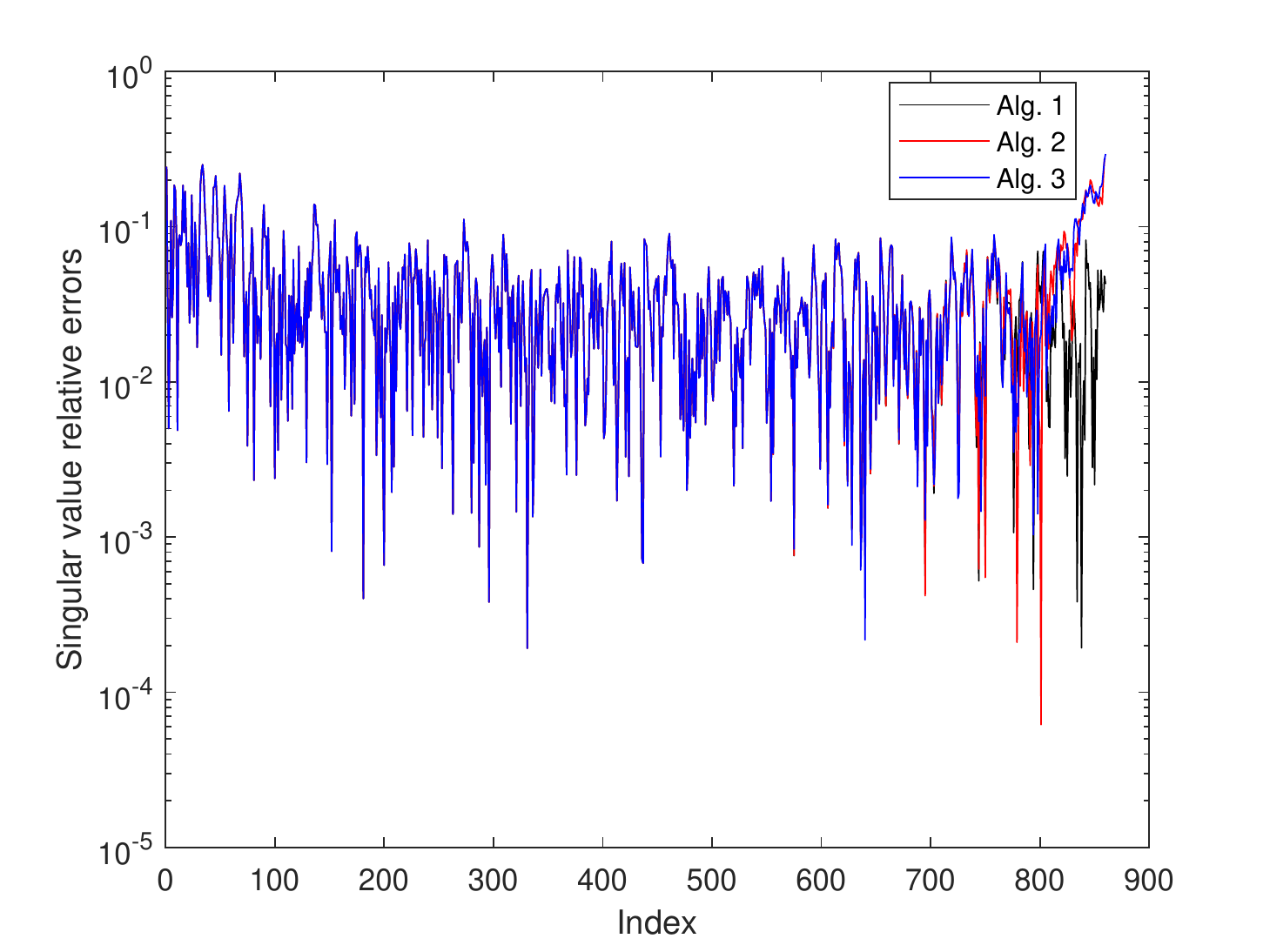}
}
\quad
\subfigure[Example \ref{ex:52}: \textsf{deriv2}]{
\includegraphics[width=6.5cm]{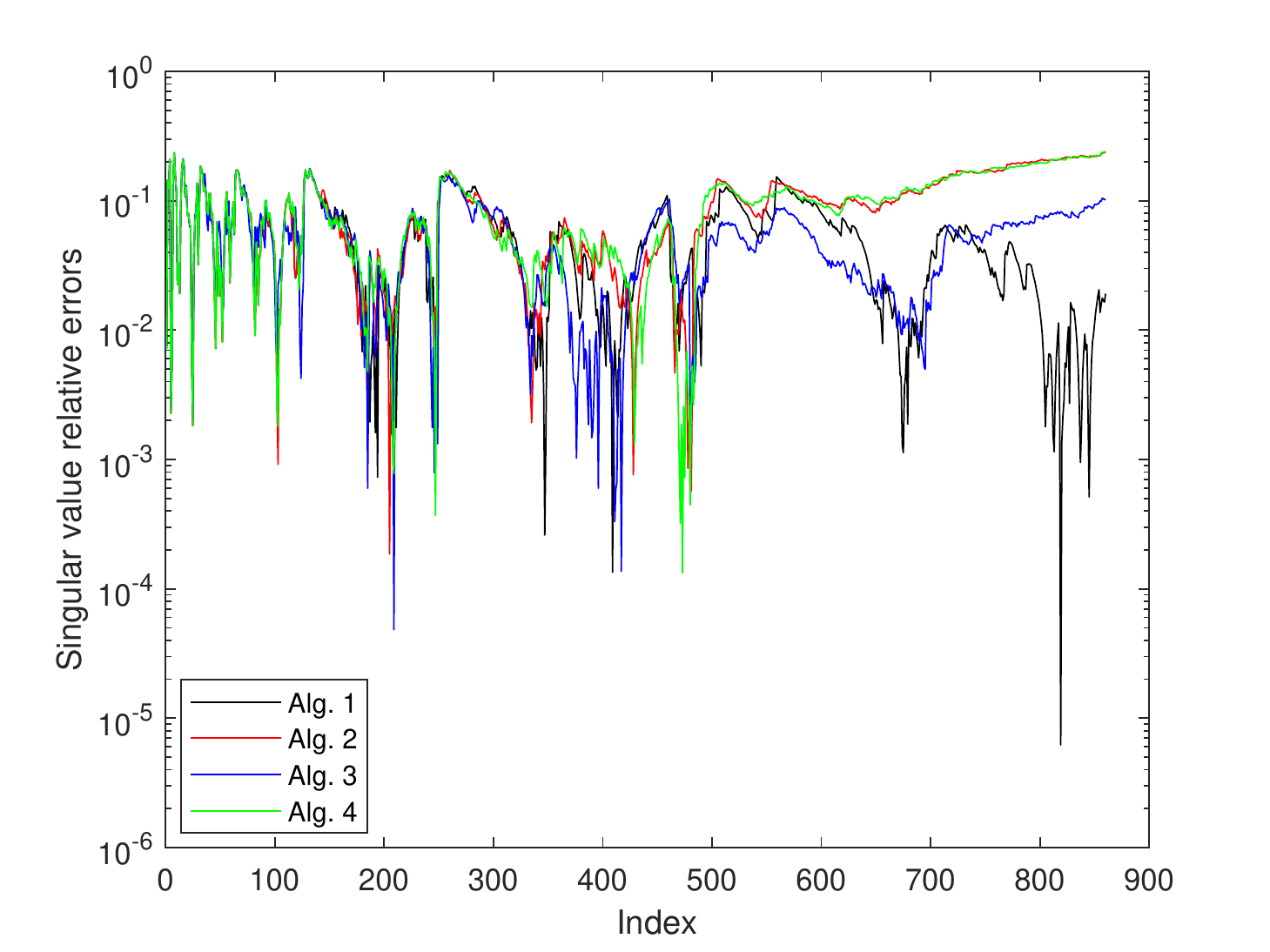}
}
\caption{Singular value relative errors for a fixed $A$.}
\label{fig:Serror_ar}
\end{figure}

\begin{figure}[H]%[htbp]
\centering
\subfigure[Example \ref{ex:51}: \textsf{pds}($t=30, s=2$)]{
\includegraphics[width=6.5cm]{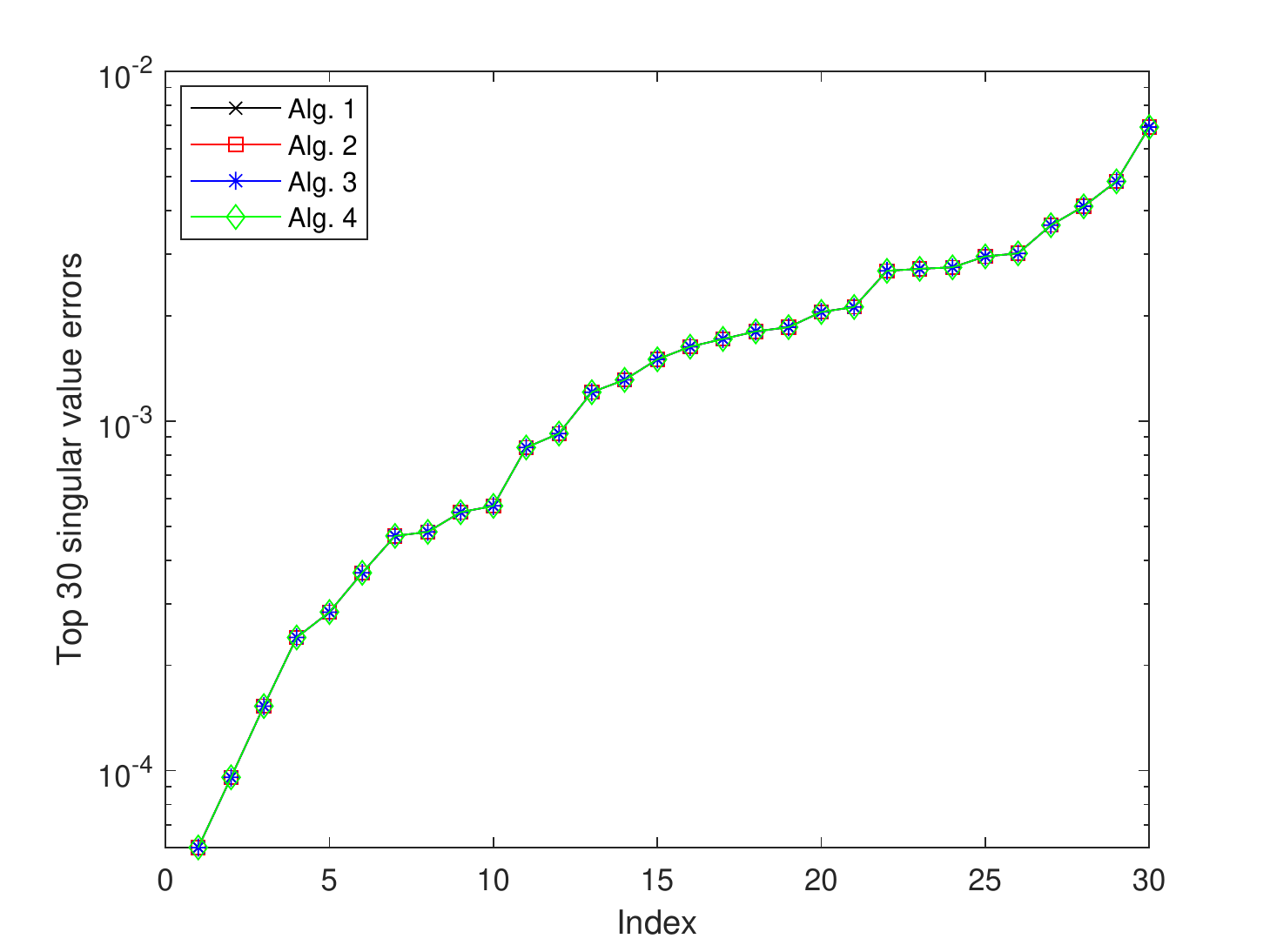}
%\caption{fig1}
}
\quad
\subfigure[Example \ref{ex:51}: \textsf{eds}($t=30, s=0.25$)]{
\includegraphics[width=6.5cm]{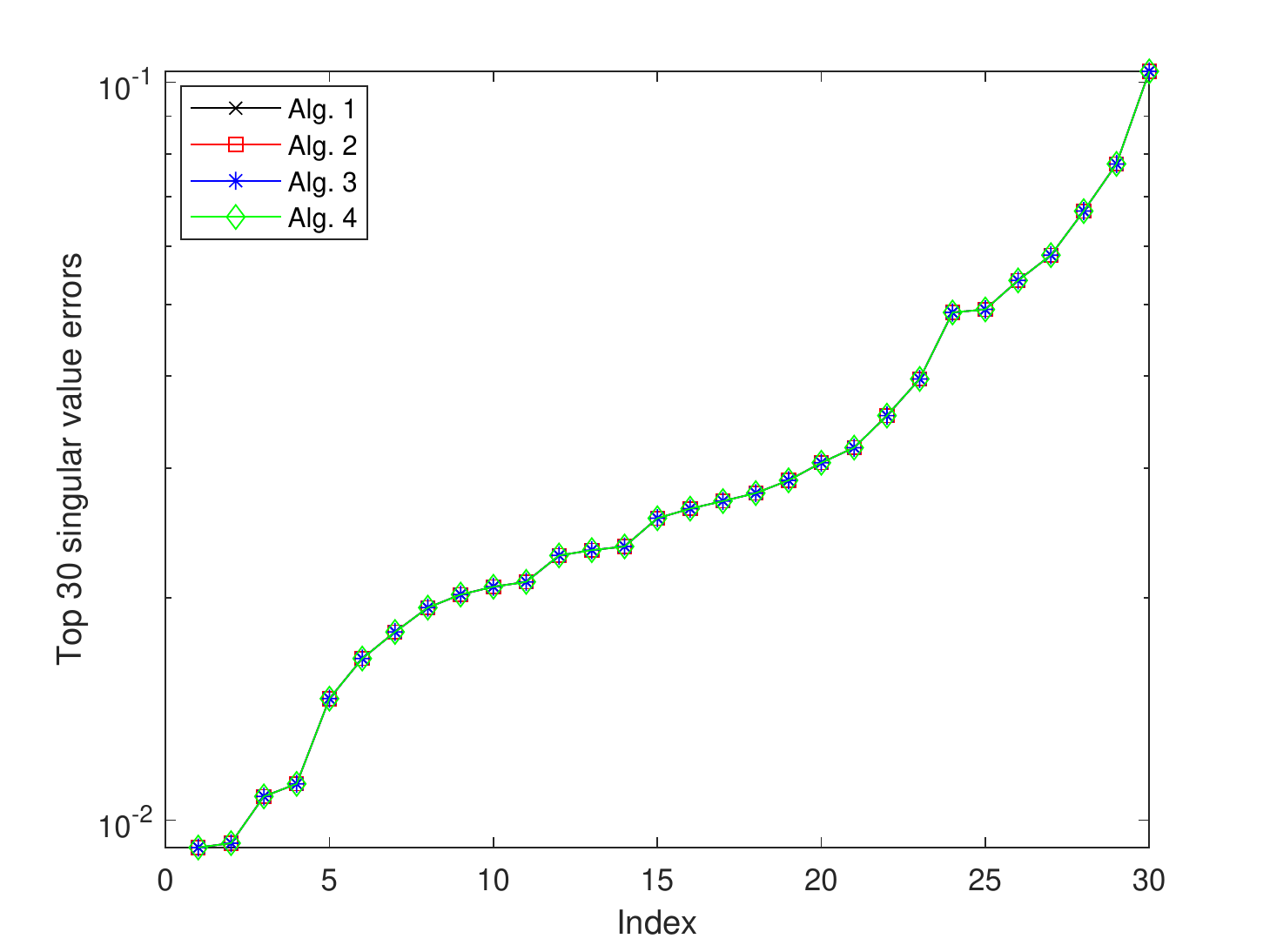}
}
\quad
\subfigure[Example \ref{ex:52}: \textsf{heat}]{
\includegraphics[width=6.5cm]{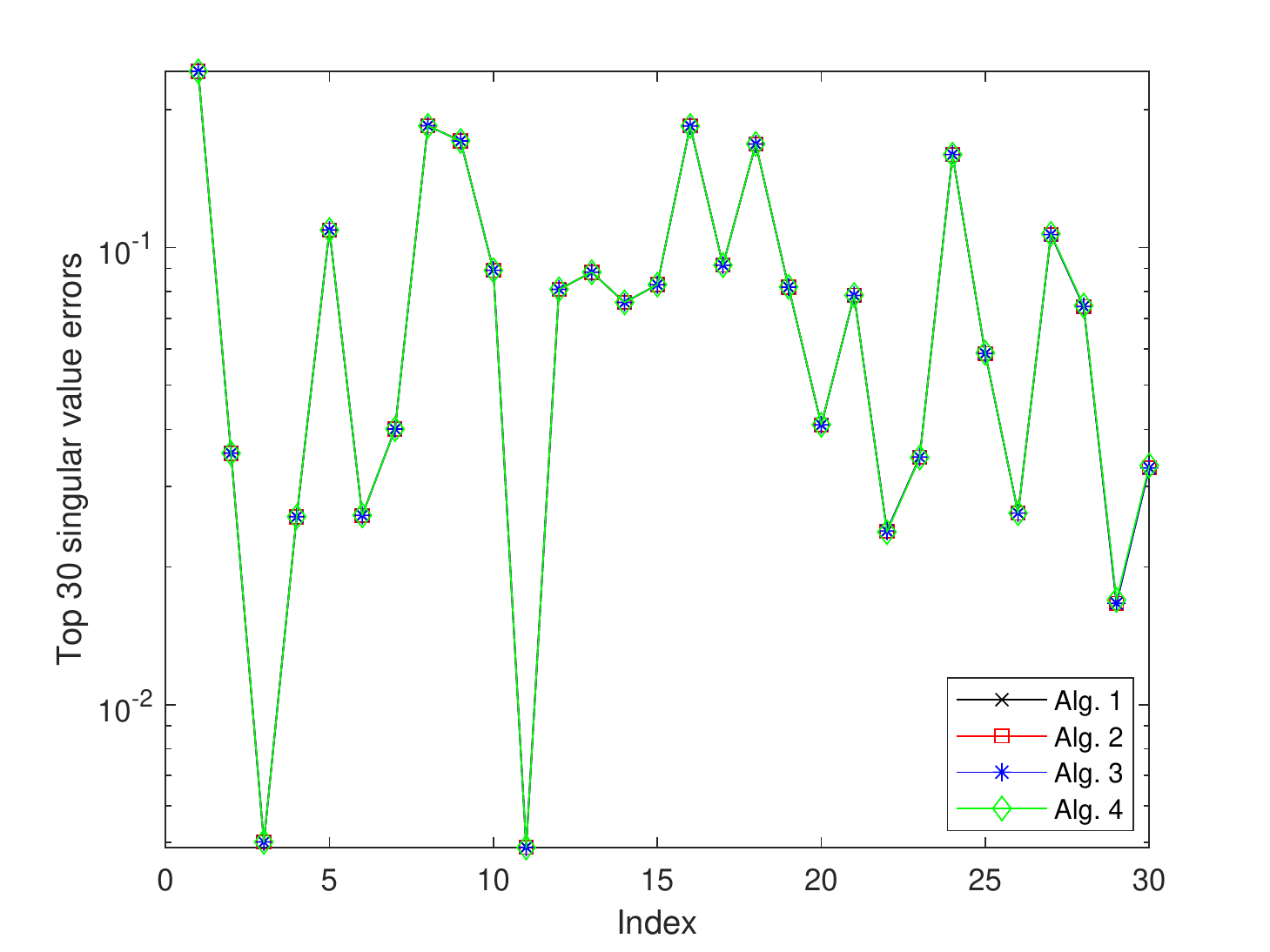}
}
\quad
\subfigure[Example \ref{ex:52}: \textsf{deriv2}]{
\includegraphics[width=6.5cm]{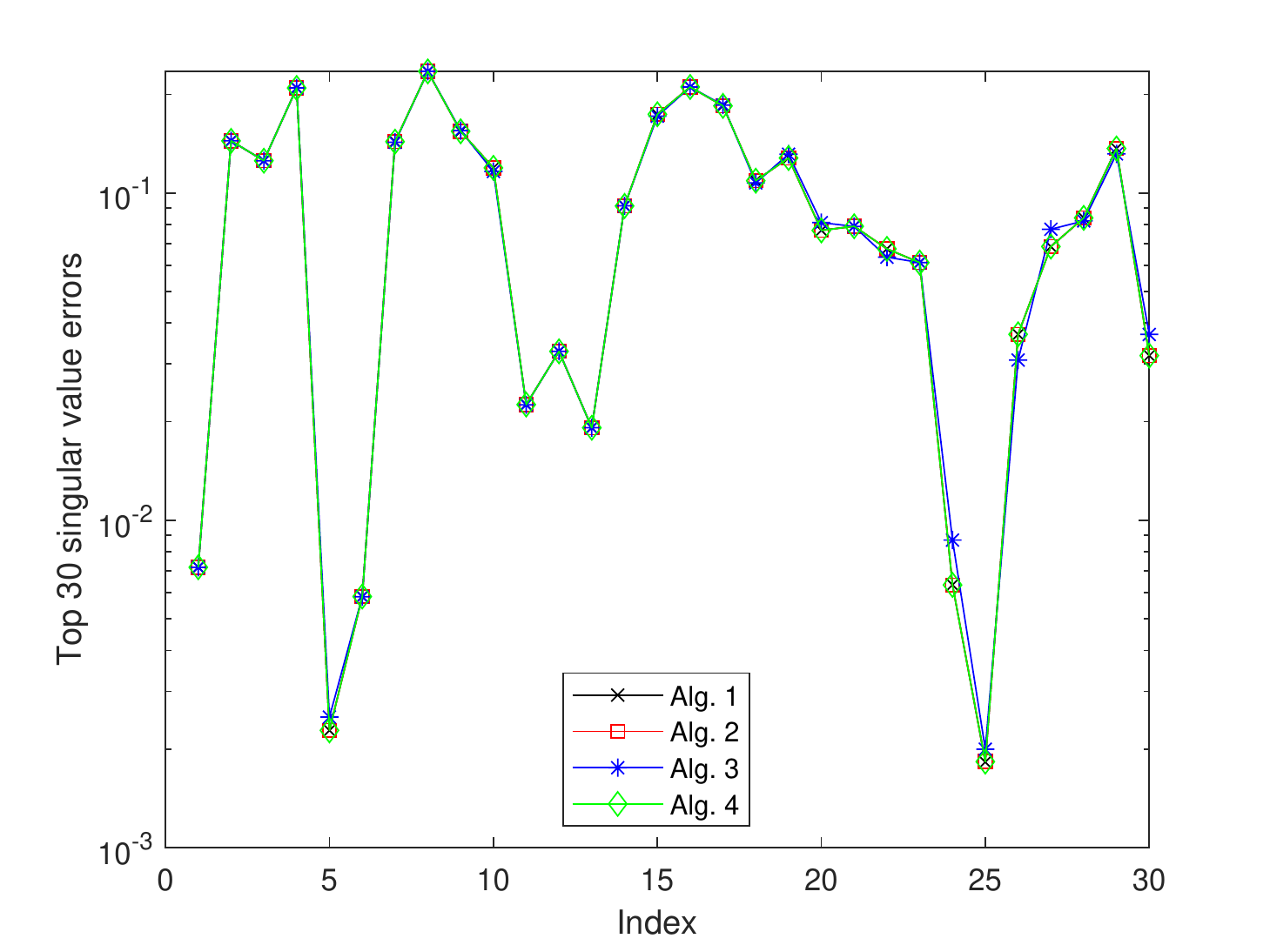}
}
\caption{Top 30 singular value relative errors for a fixed $A$.}
\label{fig:Serror_top}
\end{figure}
\section{Conclusions}
In this paper, we have proposed two single-pass randomized QLP decomposition algorithms for the low-rank approximation computing. These algorithms provide low-rank approximation of a matrix as the truncated SVD. We also give the  bounds for the matrix approximation error and the singular value approximation error, which hold with high probability. Numerical experiments also show that the two single-pass randomized  QLP decomposition algorithms have less computational cost and can achieve a desired accuracy.

%\bibliographystyle{elsarticle-num}
%\bibliography{ref}

%---------------------------------------------
\end{document}